\documentclass[a4paper,10pt,reqno]{amsart}

\usepackage{graphicx,amsmath,amssymb}
\usepackage{amssymb}
\usepackage{mathrsfs}
\usepackage{subfigure}

\newtheorem{theorem}{Theorem}
\newtheorem{lemma}{Lemma}
\newtheorem{definition}{Definition}
\newtheorem{remark}[theorem]{Remark}
\newtheorem{proposition}[theorem]{Proposition}

\let\pa\partial
\newcommand{\T}{\mathcal{T}}

\let\r\rho

\usepackage{paralist}

\title{Derivation and analysis of continuum models for crossing pedestrian traffic} 

\author{Sabine Hittmeir$^*$}

\author{Helene Ranetbauer$^\ddagger$}

\author{Christian Schmeiser$^\S$}

\author{Marie-Therese Wolfram$^\P$}

\begin{document}
\maketitle
\renewcommand{\thefootnote}{\fnsymbol{footnote}}
\footnotetext[1]{University of Vienna, Faculty for Mathematics, Oskar-Morgenstern-Platz 1, 1090 Wien, Austria.\\
sabine.hittmeir@univie.ac.at}
\footnotetext[3]{Radon Institute for Computational and Applied Mathematics, 
Austrian Academy of Sciences, Altenberger Strasse 69, 4040 Linz, Austria.\\
helene.ranetbauer@ricam.oeaw.ac.at}
\footnotetext[4]{University of Vienna, Faculty for Mathematics, Oskar-Morgenstern-Platz 1, 1090 Wien, Austria\\
christian.schmeiser@univie.ac.at}
\footnotetext[5]{Mathematics Institute, University of Warwick, Coventry CV4 7AL, UK and Radon Institute for Computational and Applied Mathematics, Austrian Academy of Sciences, Altenberger Strasse 69, 4040 Linz, Austria.\\
m.wolfram@warwick.ac.uk, mt.wolfram@ricam.oeaw.ac.at}

\renewcommand{\thefootnote}{\arabic{footnote}}

\pagestyle{myheadings}
\thispagestyle{plain}
\markboth{Hittmeir, Ranetbauer, Schmeiser, Wolfram}{Crossing pedestrian traffic}

\begin{abstract}
In this paper we study hyperbolic and parabolic nonlinear partial differential equation models, which describe the evolution of two intersecting pedestrian flows. We assume that individuals avoid collisions by sidestepping, which is encoded in the transition rates of the microscopic 2D model. We formally derive the corresponding mean-field models and prove existence of global weak solutions for the parabolic model. Moreover we discuss stability of stationary states for the corresponding one-dimensional model. Furthermore we illustrate the rich dynamics of both systems with numerical simulations. 
\end{abstract}

\bigskip

\begin{small}
\textbf{Keywords:} Crossing pedestrian traffic; segregation; stability analysis.
\end{small}

\begin{small}
\textbf{AMS Subject Classification:} 35K65, 35K55, 35A01, 35B35
\end{small}

\section{Introduction}

The complex dynamics of large pedestrian crowds attracted the attention of researchers in various scientific fields over the last decades. Starting with empirical observations in the early 1950s, pedestrian research has become an active area of research in physics, transportation processes, computer science and applied mathematics. Especially understanding and modeling the complex interactions among pedestrians has gained importance due to the ongoing development of software packages, which are used increasingly in the design and evaluation of public facilities and environments. \\
The proposed mathematical models describe the dynamics of large pedestrian flows on different levels: either microscopically by considering the motion of each individual or macroscopically by studying the evolution of the overall density distribution. The most prominent microscopic approaches are the Social Force model, see Ref. \cite{HelbingMolnar1995} and cellular automata models, see Ref. \cite{kirchner2002simulation} or Ref. \cite{blue2001cellular}. On the macroscopic level different nonlinear PDE systems, see for example Ref. \cite{carrillo2016improved}, Ref. \cite{burger2016lane}, Ref. \cite{cristiani2011multiscale}, Ref. \cite{maury2010macroscopic}, Ref. \cite{colombo2005pedestrian}, have been proposed to describe the dynamics of the pedestrian density usually based on nonlinear conservation laws. More recently also kinetic and multi-scale models have been used for example to analyze the interactions between the dynamics and social behaviors, cf. Ref. \cite{bellomo2015toward} or to model the interactions of large groups with a certain number of leaders, see Ref. \cite{degond2013vision}. In Ref. \cite{bellomo2013microscale}, it is shown how the interactions at the microscopic scale are transferred to the macroscopic one. In addition we also want to mention a recent paper dealing with a hybrid (macroscopic and kinetic) model for crowd dynamics, where the parameters related to the speed are induced by an emotional state, see Ref. \cite{wang2016efficient}. For a detailed overview on mathematical models for pedestrian dynamics we refer to Ref. \cite{bellomodogbe2011}, Ref. \cite{bellomo2} and Ref. \cite{cristiani2014multiscale}.

In this work we consider two groups - called red and blue individuals - which move from the left to the right and the bottom to the top respectively. Each individual tries to move in its desired direction (either to the right or towards the top), but steps aside to avoid collisions with the other group. We start with a lattice based approach and (formally) derive, analyze and simulate the corresponding PDE systems describing the evolution of these crossing pedestrian flows. The side-stepping behavior results in the formation of complex patterns on the microscopic as well as the macroscopic level. In the case of bidirectional flows, that is two groups walking in opposite direction, we observe the formation of directional lanes, see Ref. \cite{burger2016lane}. In the case of intersecting flows the groups segregate, forming stationary and transient diagonal patterns at the intersection. Similar patterns have been observed on the microscopic level in Ref. \cite{cividini2013diagonal,cividini2013crossing} and in a kinetic model proposed in Ref. \cite{festa2016}. We study the dynamic properties of solutions to the derived PDE models, which are either parabolic or hyperbolic (in certain density regimes). While the parabolic PDE model has a perturbed gradient flow structure, which can be analyzed using similar techniques as proposed in Ref. \cite{burger2016lane}, the PDE system derived by considering the expansion up to order one is hyperbolic in the x and the y direction only. Hence we study a 1D reduction, which has a similar structure and behavior as the models analyzed in Ref. \cite{chertock2014pedestrian} and in Ref. \cite{AppertRolland2011}, and obtain linear stability and local $L^2$ stability in certain density regimes (namely where the overall density is not too high). We would like to mention that related results have been shown for classic traffic flow models for $n$ populations, see Ref. \cite{benzoni2003}.

This paper is organized as follows: we introduce the microscopic modeling setup and the corresponding PDE system in Section \ref{s:mod}. Then we show global in time existence for the full 2D parabolic system in Section \ref{existence} and illustrate the behavior of the model with micro- and macroscopic simulations. In Section \ref{1dmodel} we discuss the dynamics of solutions to a reduced 1D hyperbolic system by studying linear stability and local $L^2$ stability behavior.

\section{Discrete and continuous models for intersecting pedestrian flows}\label{s:mod}

\subsection{A stochastic individual based model on a two-dimensional lattice} \label{individual_model}
We consider an equidistant grid of mesh size $h$ on a periodic box represented by $\Omega = [0,Nh]^2\subseteq \mathbb{R}^2$, where $x=0$ is identified with $x=Nh$, and $y=0$ with $y=Nh$. Each lattice site $(x_i, y_j) = (ih, jh)$, $i,j\in \{0,\ldots N\}$, can be empty, or it can be occupied by either a red or a blue individual. We also introduce discrete times $t_k = k\Delta t$, $k=0,1,\ldots$ with time step $\Delta t$.
The discrete stochastic processes $r^k = (r_{i,j}^k,\,i,j=0,\ldots,N)$, $b^k = (b_{i,j}^k,\,i,j=0,\ldots,N)$, where $r_{i,j}^k, b_{i,j}^k\in \{0,1\}$ with the constraint $\rho_{i,j}^k := r_{i,j}^k + b_{i,j}^k\le 1$
 indicate, if at time $t_k$ the site $(x_i,y_j)$ is occupied by a red individual ($r_{i,j}^k = 1$, $b_{i,j}^k = 0$),
or by a blue individual ($r_{i,j}^k = 0$, $b_{i,j}^k = 1$), or if it is empty ($r_{i,j}^k = b_{i,j}^k = 0$).

The general movement direction for the red individuals is to the right, i.e. in the positive $x$-direction, and for
the blue individuals upwards, i.e. in the positive $y$-direction. Every individual might also step to the side, 
in particular when a forward step is inhibited by a member of the other group. For the red individuals this leads
to the transition probabilities
\begin{align}\label{transprob_r}
\T_r^{i,j\rightarrow i+1,j}(r,b)&= \alpha(1-\rho_{i+1,j}),\nonumber\\
\T_r^{i,j\rightarrow i,j-1}(r,b)&= \alpha(1-\rho_{i,j-1})(\gamma_0+\gamma_1 \, b_{i+1,j}),\nonumber\\
\T_r^{i,j\rightarrow i,j+1}(r,b)&= \alpha(1-\rho_{i,j+1})(\gamma_0+\gamma_2 \, b_{i+1,j}),
\end{align}
with 
$$
\T_r^{i,j\rightarrow i,j} = 1 - \T_r^{i,j\rightarrow i+1,j} - \T_r^{i,j\rightarrow i,j-1} - \T_r^{i,j\rightarrow i,j+1}
$$
and $\T_r^{i,j\to m,n}=0$ for all other $(m,n)$. The nonnegative parameters $\alpha,\gamma_0,\gamma_1,\gamma_2$
satisfy 
\begin{equation}\label{CFL}
 \alpha \max\{1, 2\gamma_0 + \gamma_1 + \gamma_2\} \le 1 \,,
\end{equation}
such that $\T_r^{i,j\rightarrow i,j}\ge 0$ always holds.
Since individuals can only jump into a cell if it is not occupied, all transition probabilities $\T_r^{i,j\to m,n}$ have the 
factor $(1-\rho_{m,n})$.
The assumption that the blue individuals have the same behavior as the red ones, leads to 
\begin{align}\label{transprob_b}
\T_b^{i,j\rightarrow i,j+1}(r,b)&=\alpha(1-\rho_{i,j+1}),\nonumber\\
\T_b^{i,j\rightarrow i-1,j}(r,b)&=\alpha(1-\rho_{i-1,j})(\gamma_0+\gamma_1 \, r_{i,j+1}),\nonumber\\
\T_b^{i,j\rightarrow i+1,j}(r,b)&=\alpha(1-\rho_{i+1,j})(\gamma_0+\gamma_2 \, r_{i,j+1}),
\end{align}
The sidestepping probability can be asymmetric, with $\gamma_1>\gamma_2$ describing a tendency to sidestep
against the general movement direction of the other group. Note that all individuals refuse to move backwards.

The stochastic process is completed by prescribing when the jumps are carried out. We use the somewhat artificial
assumption of complete synchronization, where all individuals use the same information on the present state for
making the next move. This means that all individuals use the transition probabilities $\T_r^{i,j\to m,n}(r^k,b^k)$ and
$\T_b^{i,j\to m,n}(r^k,b^k)$ to determine their positions at time $t_{k+1}$.

\subsection{A discrete compartment model}

A related model is based on the assumption that each grid point $(x_i,y_j)$ represents a compartment, possibly
containing many individuals. Now $r_{i,j}^k, b_{i,j}^k\in [0,1]$ denote the fractions of the total available space in
a compartment occupied at time $t_k$ by red and, respectively, blue individuals (again with the obvious restriction 
$r_{i,j}^k+b_{i,j}^k\le 1$). We postulate the deterministic dynamics given by
\begin{align}\label{master_r}
r_{i,j}^{k+1}&=\left(1 - \T_r^{i,j\to i+1,j}(r^k,b^k) - \T_r^{i,j\to i,j-1}(r^k,b^k) - \T_r^{i,j\to i,j+1}(r^k,b^k)\right)r_{i,j}^k\nonumber \\
&\quad+\T_r^{i-1,j\to i,j}(r^k,b^k) r_{i-1,j}^k + \T_r^{i,j+1\to i,j}(r^k,b^k) r_{i,j+1}^k\\
&\quad  +\T_r^{i,j-1\to i,j}(r^k,b^k) r_{i,j-1}^k \,,\nonumber
\end{align}
\begin{align}\label{master_b}
b_{i,j}^{k+1}&=\left(1 - \T_b^{i,j\to i,j+1}(r^k,b^k) - \T_b^{i,j\to i-1,j}(r^k,b^k) - \T_b^{i,j\to i+1,j}(r^k,b^k)\right)b_{i,j}^k\nonumber\\
&\quad+\T_b^{i,j-1\to i,j}(r^k,b^k) b_{i,j-1}^k + \T_b^{i-1,j\to i,j}(r^k,b^k) b_{i-1,j}^k \\
&\quad + \T_b^{i+1,j\to i,j}(r^k,b^k) b_{i+1,j}^k \,,\nonumber
\end{align}
with the transition rates \eqref{transprob_r}, \eqref{transprob_b}. 

In principle, a connection could be made between the two models by passing to expectation values in the 
stochastic model, but the expected transition rates will not satisfy \eqref{transprob_r}, \eqref{transprob_b}
because of the expected strong local correlations.

\subsection{The macroscopic PDE model}\label{derivation_macro}
A continuous model, both in position and time, can be obtained from \eqref{master_r}, \eqref{master_b} by 
interpreting $r_{i,j}^k$ and $b_{i,j}^k$ as approximations for the values $r(x_i,y_j,t_k)$ and $b(x_i,y_j,t_k)$ of continuous functions, and formally
passing to the limit $h,\Delta t\to 0$ (similarly to Ref. \cite{burger2016lane}). We assume that position and time have already been non-dimensionalized and
make the additional assumption $\alpha = \Delta t/h$. Assumption \eqref{CFL} can then be interpreted as a CFL-condition. Division of \eqref{master_r}, \eqref{master_b} by $\Delta t$ and passing to the limit leads to

\begin{align}\label{e:hypsys}
\pa_t r &+\pa_x ((1-\r)r) + (\gamma_2-\gamma_1) \pa_y ((1-\r) br) = 0 \,,\nonumber\\
\pa_t b &+\pa_y ((1-\r)b)  + (\gamma_2-\gamma_1) \pa_x ((1-\r) br) = 0 \,.
\end{align}

The second terms on the left hand sides correspond to the motion in the walking direction (to the right and, respectively, upwards), while the third terms correspond to the side-stepping behavior. 
A natural regularization is obtained by carrying out the Taylor expansions to the second order with respect to position in the right hand
sides of \eqref{master_r}, \eqref{master_b}, to the next order, leading to
\begin{align}\label{system}
\pa_t r & + \nabla \cdot J_r = 0 \,,\nonumber\\
\pa_t b &+\nabla \cdot J_b = 0 \,,
\end{align}
where the flows of red and blue individuals are given by
\[J_r:=\begin{pmatrix}
(1-\r)r-\varepsilon\left[(1-\rho)\pa_x r+r\pa_x \rho\right]\\\\
-(\gamma_1-\gamma_2)(1-\rho)b r-\varepsilon\left[(\gamma_1+\gamma_2)\left((1-\r)\pa_y(rb)+br\pa_y\r\right) \right.\\
\left.+2\gamma_0\left((1-\rho)\pa_y r +r\pa_y\rho\right)+2(\gamma_1-\gamma_2)(1-\r)r\pa_xb  \right]\\
\end{pmatrix},\]
and
\[J_b:=\begin{pmatrix}
-(\gamma_1-\gamma_2)(1-\rho)b r-\varepsilon\left[(\gamma_1+\gamma_2)\left((1-\r)\pa_x(rb)+br\pa_x\r\right)\right.\\
\left. +2\gamma_0\left((1-\rho)\pa_x b +b\pa_x\rho\right)+2(\gamma_1-\gamma_2)(1-\r)b\pa_y r  \right]\\\\
(1-\r)b-\varepsilon\left[(1-\rho)\pa_y b+b\pa_y \rho\right]\\
\end{pmatrix},\]
with $\varepsilon=h/2$. This regularization is related to 'modified equations' as used in numerical analysis to 
understand the qualitative behavior of numerical schemes. However, we have neglected the Taylor expansion
with respect to time, which corresponds to a smallness assumption on the parameter $\alpha$ (which can be
interpreted as a Courant number).

\section{Global existence of the parabolic problem}\label{existence}
In this section, we prove global in time existence for the second order parabolic problem. We consider the system \eqref{system} on $\Omega \times (0,T)$, where $\Omega \subseteq \mathbb{R}^2$ is bounded. In the following, we set w.l.o.g. $\Omega=[0,1]\times[0,1]$. We assume the system to be supplemented with periodic boundary conditions and for simplicity set 
\[\gamma:=\gamma_1=\gamma_2.\] 
All arguments also hold in the case $\gamma_1\neq \gamma_2$, where $\gamma_1-\gamma_2$ is sufficiently small. For details see Remark \ref{remark3.1}. 
In a similar fashion to Ref. \cite{burger2016lane} we define the entropy functional 
\begin{align}\label{entropy}
E:= \varepsilon\int_{\Omega}& r(\log r-1)+b(\log b-1)\,dx\,dy \nonumber\\
&+\varepsilon\int_{\Omega}\frac{1}{2}(1-\rho)(\log (1-\rho)-1)\,dx\,dy+\int_{\Omega} r V_r+ bV_b \, dx \,dy,
\end{align}
where the potentials $V_r(x,y)=-x$ and $V_b(x,y)=-y$ correspond to the motion of the red and blue individuals to the right and the top respectively. As we shall see below, this functional is not an entropy in the strict sense that it is decaying for all times for any solution. Instead it increases at most linearly in time, which still allows us to prove global existence of weak solutions. Introducing the corresponding entropy variables   
\begin{equation*} 
u:=\pa_r E=\varepsilon\log r -\frac{\varepsilon}{2}\log(1-\rho)+ V_r \quad \text{and} \quad v:=\pa_b E=\varepsilon\log b -\frac{\varepsilon}{2}\log(1-\rho)+ V_b,
\end{equation*}
allows us to write system \eqref{system} as
\begin{align}
\label{system_entropy}
\begin{pmatrix}
\pa_t r \\ \pa_t b
\end{pmatrix} &=\begin{pmatrix}
\nabla & 0\\ 0 & \nabla 
\end{pmatrix} \cdot\left( M(r,b)\begin{pmatrix}
\pa_x u\\ \pa_y u \\ \pa_x v \\ \pa_y v
\end{pmatrix}
+\varepsilon\begin{pmatrix}
\frac{r}{2}\pa_x \rho \\ \gamma_0 r \pa_y \rho \\ \gamma_0 b \pa_x \rho \\ \frac{b}{2}\pa_y \rho 
\end{pmatrix}+\begin{pmatrix}
0 \\ 2\gamma rb (1-\rho)\\ 2\gamma rb (1-\rho) \\0
\end{pmatrix}\right),
\end{align}
where
\[M:=M(r,b)=(1-\rho)\begin{pmatrix}
r & 0 & 0 & 0\\
0 & 2r(\gamma_0 +\gamma b) & 0 & 2\gamma rb \\
2\gamma rb & 0 & 2b(\gamma_0 +\gamma r) & 0\\
0 & 0 & 0 & b
\end{pmatrix}.\]

Note that system \eqref{system_entropy} has a similar structure as the PDE model for bidirectional flow studied in Ref. \cite{burger2016lane}. Hence we can use similar arguments to prove existence which we briefly state in the following.\\

We start by showing that the entropy is growing at most linearly in time.

\begin{lemma}\label{lemma_entropy}
Let $r, b :\Omega \rightarrow \mathbb{R}^2$ be a sufficiently smooth solution to system \eqref{system_entropy} for $\frac{1}{8}< \gamma_0 < 1$ satisfying 
\begin{align*}
0\leq r,b \quad \text{ and }\quad \rho\leq 1.
\end{align*}
Then there exists a constant $C\geq 0$ such that
\begin{align}
\label{entropyinequality}
\begin{aligned}
\frac{\mathrm{d}E}{\mathrm{d}t}+\mathcal{D}_0&\leq C,
\end{aligned}
\end{align}
where 
\begin{equation*}
\mathcal{D}_0=\varepsilon^2 C_0\int_{\Omega} (1-\rho)|\nabla\sqrt{r}|^2 +(1-\rho)|\nabla\sqrt{b}|^2+ |\nabla\sqrt{1-\rho}|^2+|\nabla \rho|^2\,dx\,dy,
\end{equation*}
for some constant $C_0>0$.
\end{lemma}
\smallskip
\begin{proof}
Using \eqref{system_entropy} we deduce the following entropy dissipation relation:
\begin{align}\label{cr_equ_1}
\begin{aligned}
\frac{\mathrm{d}E}{\mathrm{d}t} &=\int_{\Omega} (u\, \pa_t r +v\, \pa_t b ) dx\,dy\\
&=-\int_{\Omega}  M \begin{pmatrix}
\nabla u \\ \nabla v
\end{pmatrix}\cdot \begin{pmatrix}
\nabla u \\ \nabla v
\end{pmatrix}+\left(\varepsilon\begin{pmatrix}
\frac{r}{2}\pa_x \rho \\ \gamma_0r\pa_y \rho\\ \gamma_0 b\pa_x \rho \\ \frac{b}{2}\pa_y \rho 
\end{pmatrix}+\begin{pmatrix}
0 \\ 2\gamma rb (1-\rho)\\ 2\gamma rb (1-\rho) \\0
\end{pmatrix}\right)\cdot \begin{pmatrix}
\nabla u \\ \nabla v
\end{pmatrix}\,dx\,dy.\\
\end{aligned}
\end{align}
In terms of the entropy variables $u$ and $v$, we can rewrite
\[\pa_x \rho=(r\pa_x u+b\pa_x v+r)\frac{2(1-\rho)}{\varepsilon(2-\rho)},\]
and
\[\pa_y \rho=(r\pa_y u+b\pa_y v+b)\frac{2(1-\rho)}{\varepsilon(2-\rho)}.\]
Hence, equation \eqref{cr_equ_1} becomes
\begin{align}\label{cr_equ_3}
\begin{aligned}
\frac{\mathrm{d}E}{\mathrm{d}t}&=-\int_{\Omega}  M \begin{pmatrix}
\nabla u \\ \nabla v
\end{pmatrix}\cdot \begin{pmatrix}
\nabla u \\ \nabla v
\end{pmatrix}+  N(r,b) \begin{pmatrix}
\nabla u \\ \nabla v
\end{pmatrix}\cdot \begin{pmatrix}
\nabla u \\ \nabla v
\end{pmatrix}+H(r,b)\cdot \begin{pmatrix}
\nabla u \\ \nabla v
\end{pmatrix} \,dx\,dy,\\
\end{aligned}
\end{align}
where 

\[N:=N(r,b)=\frac{1-\rho}{2-\rho}\begin{pmatrix}
r^2& 0 & rb & 0\\
0 & 2\gamma_0r^2 & 0 & 2\gamma_0 rb \\
2\gamma_0 rb & 0 & 2\gamma_0 b^2 & 0\\
0 & rb & 0 & b^2
\end{pmatrix},\]
and
\[H:=H(r,b)=\frac{1-\rho}{2-\rho}\begin{pmatrix}
r^2\\ 2\gamma_0 rb\\ 2\gamma_0 rb\\b^2
\end{pmatrix}+\begin{pmatrix}
0 \\ 2\gamma rb (1-\rho)\\ 2\gamma rb (1-\rho) \\0
\end{pmatrix}.\]

Hence, 
\begin{align}\label{cr_equ_4}
\frac{\mathrm{d}E}{\mathrm{d}t}&=-\int_{\Omega} \bigg[ (1-\rho)(r(\pa_x u)^2 +b(\pa_y v)^2) +2\gamma_0(1-\rho)(r(\pa_y u)^2 +b(\pa_x v)^2)\nonumber\\
&\qquad\qquad +2\gamma(1-\rho)rb((\pa_y u)^2+(\pa_x v)^2+\pa_y v\pa_y u+\pa_xu\pa_x v)\bigg]\,dx\,dy\nonumber\\
&\quad -\int_{\Omega} \frac{1-\rho}{2-\rho}\bigg[r^2(\pa_x u)^2+2\gamma_0r^2(\pa_y u)^2+2\gamma_0b^2(\pa_x v)^2+b^2(\pa_y v)^2\nonumber\\
&\quad +rb\pa_x u\pa_x v+2\gamma_0 rb\pa_y u\pa_y v+2\gamma_0 rb \pa_x u\pa_x v+rb\pa_yu\pa_yv \bigg]\,dx\,dy\nonumber\\
&\quad -\int_{\Omega}H(r,b)\cdot \begin{pmatrix}
\nabla u \\ \nabla v
\end{pmatrix} \,dx\,dy.
\end{align}
As $0\leq \gamma,r,b,\rho\leq 1$, we have 
\[2\gamma(1-\rho)rb|\pa_x u \pa_x v|\leq \frac{1-\rho}{2}r(\pa_x u)^2 +2\gamma (1-\rho)rb(\pa_x v)^2,\]
\[ \frac{1-\rho}{2-\rho}rb|\pa_x u\pa_x v|\leq \frac{1-\rho}{2-\rho}r^2 (\pa_x u)^2+\frac{1-\rho}{4}b(\pa_x v)^2,\]
\[ \frac{1-\rho}{2-\rho}2\gamma_0 rb |\pa_x u\pa_x v|\leq 2\gamma_0 \frac{1-\rho}{4}r(\pa_x u)^2+\frac{1-\rho}{2-\rho}2\gamma_0 b^2 (\pa_x v)^2,\]
by Young's inequality. The same holds for the term involving $y$-derivatives. In order to guarantee that all the mixed terms are controlled by the quadratic terms, we have to assume that $\frac{1}{8}<\gamma_0<1$. Hence, the entropy dissipation \eqref{cr_equ_1} reduces to
\begin{equation*}
\frac{\mathrm{d}E}{\mathrm{d}t}\leq   -\int_{\Omega} \tilde{C}(1-\rho)(r|\nabla u|^2+b |\nabla v|^2)+H(r,b)\cdot \begin{pmatrix}
\nabla u \\ \nabla v
\end{pmatrix} \,dx\,dy,
\end{equation*}
where $\tilde{C}=\min(2\gamma_0-\frac{1}{4},\frac{1}{2}(1-\gamma_0))>0$.\\
The linear terms, i.e. the terms arising from the matrix $H(r,b)$, can be controlled by Young's inequality resulting in at most linear growth of the entropy functional. In particular, the first term can be bounded by
\[\int_{\Omega}\frac{(1-\rho)r^2}{2-\rho}|\pa_x u|\,dx\,dy \leq \int_{\Omega} \frac{\tilde{C}}{8}(1-\rho)r (\pa_x u)^2\,dx\,dy +\int_{\Omega}\frac{2(1-\rho)r^2}{\tilde{C}}\,dx\,dy.\]
As $0\leq r,\rho \leq 1$ and the integration is over a bounded domain, we obtain
\begin{equation*}
\frac{\mathrm{d}E}{\mathrm{d}t}\leq   -\int_{\Omega} \frac{\tilde{C}}{2}(1-\rho)(r|\nabla u|^2+b |\nabla v|^2)+\hat{C} \,dx\,dy,\\
\end{equation*}
for some constant $\hat{C}\geq 0$.\\

Using the definitions of $u$ and $v$, applying Young's inequality to estimate the mixed terms involving the potentials as well as the fact that
\begin{align*}
&r(1-\rho)\left|\nabla\left(\log\frac{r}{\sqrt{1-\rho}}\right)\right|^2+b(1-\rho)\left|\nabla\left(\log\frac{b}{\sqrt{1-\rho}}\right)\right|^2\\
=\,&4(1-\rho)\left|\nabla\sqrt{r}\right|^2+4(1-\rho)\left|\nabla\sqrt{b}\right|^2+\rho\left|\nabla\sqrt{1-\rho}\right|^2+\left|\nabla\rho\right|^2,
\end{align*}
we obtain 
\begin{align*}
\frac{\mathrm{d}E}{\mathrm{d}t}&\leq-\frac{\tilde{C}\varepsilon^2}{4}\int_{\Omega} 4(1-\rho)|\nabla\sqrt{r}|^2 +4(1-\rho)|\nabla\sqrt{b}|^2+\rho |\nabla\sqrt{1-\rho}|^2+|\nabla \rho|^2\,dx\,dy\\
&\quad +\frac{\tilde{C}}{2}\int_{\Omega} (1-\rho)(r|\nabla V_r|^2 +b|\nabla V_b|^2)\,dx\,dy+\hat{C}
\end{align*}
Since $|\nabla V_r|^2=|\nabla V_b|^2=1$ and $\rho |\nabla\sqrt{1-\rho}|^2+|\nabla \rho|^2\geq |\nabla\sqrt{1-\rho}|^2,$
we get the estimate

\begin{align}
\label{lf_entropy_ineq2}
\begin{aligned}
\frac{\mathrm{d}E}{\mathrm{d}t}&\leq -\varepsilon^2 C_0\int_{\Omega} (1-\rho)(|\nabla\sqrt{r}|^2 +|\nabla\sqrt{b}|^2)+|\nabla\sqrt{1-\rho}|^2+|\nabla \rho|^2\,dx\,dy+C,
\end{aligned}
\end{align}

for some constant $C\geq 0$ and $C_0=\frac{\tilde{C}}{8}$, which concludes the proof.
\end{proof}

Note that we cannot use the maximum principle to prove nonnegativity and boundedness of $r$, $b$ and $\rho$. Therefore we define the entropy density on the set
\begin{equation}\label{equ:set}
\mathcal{M}=\left\{\begin{pmatrix}
r\\b
\end{pmatrix}\in \mathbb{R}^2:r>0,b>0,r+b<1\right\},
\end{equation}
i.e.
\begin{align*}
h_E&:\mathcal{M}\to\mathbb{R},\\
\begin{pmatrix}
r\\b
\end{pmatrix}\mapsto \varepsilon \bigg(r(\log r-1) + b (\log b-1) &+ \frac{1}{2} (1-\rho)(\log (1-\rho)-1)\bigg)+ r V_r+ bV_b.
\end{align*}
The fact that the corresponding gradient is invertible yields positivity and the appropriate bounds for $r,b$ and $\rho$ (cf. Ref.~\cite{jungel2015boundedness}, \cite{burger2016lane}).\\
The previous propositions and definitions allow us to prove global existence of weak solutions.
\begin{definition}\label{def1}
A function  $(r,b):(0,T) \times \Omega \to \overline{\mathcal{M}}$ is called a weak solution to system \eqref{system_entropy} if it satisfies the formulation 
\begin{align*}
\int_0^T \int_\Omega \begin{pmatrix}
\pa_t r \\ \pa_t b
\end{pmatrix}&\cdot\begin{pmatrix}
\Phi_1\\\Phi_2
\end{pmatrix}+\varepsilon 
\begin{pmatrix}
\pa_x r(1-\rho)+r \pa_x \rho\\ \pa_y b (1-\rho)+b\pa_y \rho
\end{pmatrix}\cdot\begin{pmatrix}
\pa_x \Phi_1\\ \pa_y \Phi_2
\end{pmatrix}\,dx\,dy\,dt\nonumber\\
&+2\varepsilon\int_0^T \int_{\Omega}\gamma_0\begin{pmatrix}
\pa_y r(1-\rho)+r\pa_y \rho\\ \pa_x b(1-\rho)+b\pa_x \rho
\end{pmatrix}\cdot\begin{pmatrix}
\pa_y \Phi_1\\ \pa_x \Phi_2
\end{pmatrix}\,dx\,dy\,dt\\
&+2\varepsilon\int_0^T \int_{\Omega}\gamma\begin{pmatrix}
\pa_y (r b)(1-\rho)+r b\pa_y \rho\\ \pa_x (r b)(1-\rho)+r b\pa_x \rho
\end{pmatrix}\cdot\begin{pmatrix}
\pa_y \Phi_1\\ \pa_x \Phi_2
\end{pmatrix}\,dx\,dy\,dt\nonumber\\
&+\int_0^T\int_{\Omega}\begin{pmatrix}
(1-\rho)r\nabla V_r\\(1-\rho)b\nabla V_b
\end{pmatrix}\cdot\begin{pmatrix}
\nabla \Phi_1\\ \nabla \Phi_2
\end{pmatrix}\,dx\,dy\,dt=0,\nonumber
\end{align*}
for all $\Phi_1,\Phi_2\in L^2(0,T;H^1(\Omega))$. 
\end{definition}

\begin{theorem}{(Global existence)}\label{theorem1}
Let $T>0$ and $(r_0,b_0):\Omega \to \mathcal{M}$, where $\mathcal{M}$ is defined by \eqref{equ:set}, be measurable initial conditions such that $h_E(r_0,b_0)\in L^1(\Omega)$. If $\frac{1}{8}<\gamma_0<1$, there exists a weak solution $(r,b):\Omega\times (0,T)\to \overline{\mathcal{M}}$ in the sense of Definition \ref{def1} to system \eqref{system_entropy} with periodic boundary conditions satisfying 
\begin{align*}
&\pa_t r,\, \pa_t b \in L^2(0,T;H^1(\Omega)'),\\
&\rho,\, \sqrt{1-\rho}\,\in L^2(0,T;H^1(\Omega )),\\
&(1-\rho)\nabla\sqrt{r},\,(1-\rho)\nabla\sqrt{b} \,\in L^2(0,T;L^2(\Omega)).
\end{align*}
Moreover, the weak solution satisfies the following entropy dissipation inequality:
\begin{align}
\label{theorem1_2}
\begin{aligned}
\frac{\mathrm{d}E}{\mathrm{d}t} +\mathcal{D}_1\leq C,
\end{aligned}
\end{align}
where
\begin{equation*}
\mathcal{D}_1=\varepsilon^2 C_0\int_{\Omega} (1-\rho)^2|\nabla\sqrt{r}|^2 +(1-\rho)^2|\nabla\sqrt{b}|^2+ |\nabla\sqrt{1-\rho}|^2+|\nabla \rho|^2\,dx\,dy
\end{equation*}
and $C_0$ and $C$ are the constants from Lemma \ref{lemma_entropy}.
\end{theorem}

Since the proof follows the lines of Ref.~\cite{burger2016lane}, we omit the details and sketch its ideas only. In the first step one considers a time discrete regularized formulation of \eqref{system_entropy}, for which existence of weak solutions is guaranteed by Lax-Milgram. Then we use Schauder's fixed point theorem, cf. Ref. \cite{browder1965fixed}, to conclude the existence result for the corresponding nonlinear problem. Finally uniform a priori estimates in the discrete time step $\tau$ arising from the discrete version of the entropy inequality and the use of a generalized Aubin-Lions lemma (cf. Ref.~\cite{zamponi2015analysis}) allow to pass to the limit $\tau \to 0$.

\begin{remark}\label{remark3.1}
If $\gamma_1\neq \gamma_2$, the additional first order term can easily be controlled by Young's inequality as it was done for the other linear terms in the proof of Lemma \ref{lemma_entropy}. The resulting additional diffusion terms can only be controlled for $\gamma_1-\gamma_2$ sufficiently small by the entropy production term in \eqref{lf_entropy_ineq2}. More precisely using
\[2\varepsilon|\gamma_1-\gamma_2|(1-\rho)r|\pa_x b \pa_y u|\leq 4\varepsilon^2|\gamma_1-\gamma_2|(1-\rho)(\pa_x \sqrt{b})^2+|\gamma_1-\gamma_2|(1-\rho)r(\pa_y u)^2\]
and
\[2\varepsilon|\gamma_1-\gamma_2|(1-\rho)b|\pa_y r \pa_x v|\leq 4\varepsilon^2|\gamma_1-\gamma_2|(1-\rho)(\pa_y \sqrt{r})^2+|\gamma_1-\gamma_2|(1-\rho)b(\pa_x v)^2,\]
the estimate can be closed for $|\gamma_1-\gamma_2|<\frac{\min\left\{2\gamma_0-\frac{1}{4},\frac{1}{2}(1-\gamma_0)\right\}}{33}$.
\end{remark}

\subsection{Numerical simulations}\label{numerical_simulations_2D}
Next we illustrate the behavior of the model \eqref{system} in spatial dimension two. The following simulations have been carried out using the COMSOL Multiphysics Package with quadratic finite elements. We consider the domain $\Omega= [0,1]\times[0,1]$ with periodic boundary conditions. The spatial mesh consists of 3258 triangles, the maximum time step in the used backward differentiation formula (BDF) method is set to $0.1$.

\subsubsection{Example I: Periodic boundary conditions.}\label{ex:1}
In our first example we assume that individuals have a small preference to step to the right, that is $\gamma_1=0.15$ and $\gamma_2=0.1$.
We set $\gamma_0=0.2$, $\varepsilon=0.05$ and the initial values to 
\begin{align}\label{initialvalues}
r_0(x,y)&=r_\infty+0.02\cos(\pi x)\sin\left(\pi y\right),\nonumber\\
b_0(x,y)&=b_\infty+0.02\sin(\pi x)\cos\left(\pi y\right),
\end{align}
with $r_\infty=b_\infty=0.4$. Figure \ref{f:ex1} illustrates the initial values $r_0,b_0$ and the solution $r_T,b_T$ to system \eqref{system} at time $T=20$. We observe the formation of shifted stationary diagonal stripes for the respective pedestrian densities.  
\begin{figure}[h!]
\begin{center}
\subfigure[$r_0$]{\includegraphics[height=45mm, width=62mm]{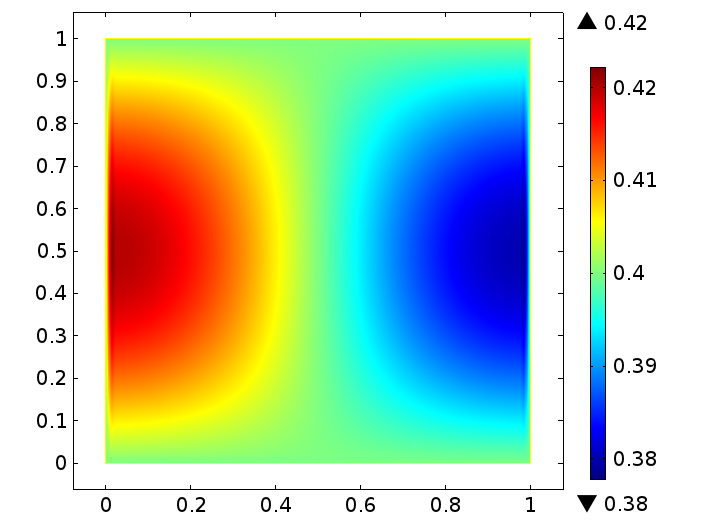}}
\subfigure[$b_0$]{\includegraphics[height=45mm, width=62mm]{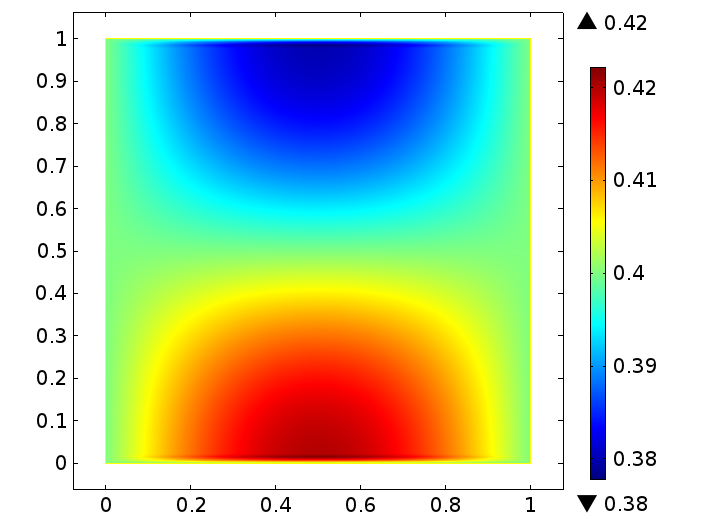}}
\subfigure[$r_T$ at $T=20$]{\includegraphics[height=45mm, width=62mm]{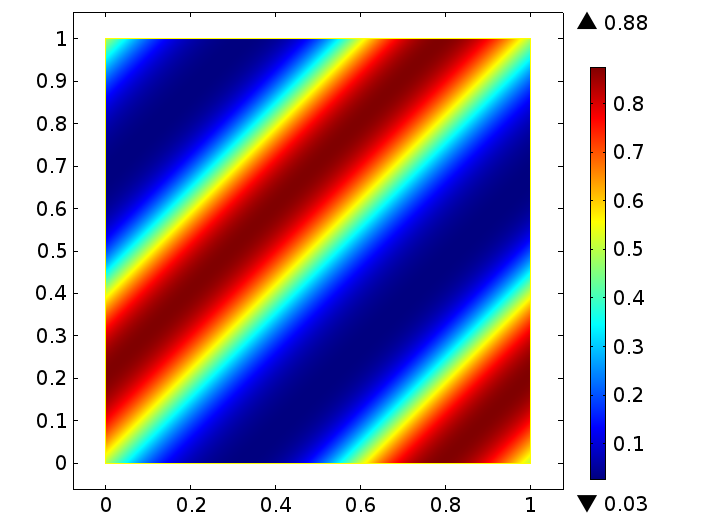}}
\subfigure[$b_T$ at $T=20$]{\includegraphics[height=45mm, width=62mm]{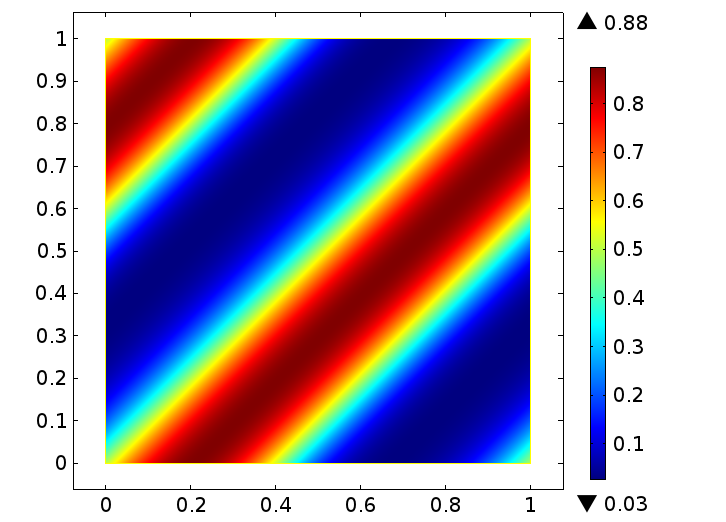}}
\caption{Example I: Formation of diagonal lanes in the red and blue particle density in the case of small perturbation of the equilibrium solutions $(r_\infty,b_\infty)$.}\label{f:ex1}
\end{center}
\end{figure}

\begin{remark}
Note that the direction of the diagonal stripes does not depend on $\gamma_1$ and $\gamma_2$. However, the total initial masses $M_r$ and $M_b$, where
\begin{equation*}
M_r:=\int_{\Omega} r_0(x,y)\,dx\,dy\quad \text{ and } \quad M_b:=\int_{\Omega} b_0(x,y)\,dx\,dy,
\end{equation*}
as well as the type of perturbation change the stationary profiles. If the total initial mass is small, perturbations smooth out quickly, the system returns to its initial equilibrium state. If the total mass is sufficiently large, as in Example \ref{ex:1}, we observe the formation of diagonal stripes. The simulations indicate that the set $\mathcal{M}$ is divided to a stable and unstable region. A rigorous proof is however left for future work. 
\end{remark}

%
%

\subsubsection{Example II: Mixed boundary conditions.}
The case of more realistic boundary conditions shows that the choice of the parameters $\gamma_1$ and $\gamma_2$ is significant. Again we start with initial values \eqref{initialvalues}, where $r_\infty=b_\infty=0.1$, and set $\gamma_0=0.15$, $\varepsilon=0.0025$. We now use the following boundary conditions:  Dirichlet  at the corresponding entrances, a given outflux at the exits and no-flux boundary conditions on the rest of the domain. More precisely
\begin{align*}
r(0,y)=b(x,0)=0.1, \quad J_r\cdot &\begin{pmatrix}
1 \\ 0
\end{pmatrix}=0.8r, \quad J_b\cdot \begin{pmatrix}
0 \\ 1
\end{pmatrix}=0.8b \\
J_r\cdot \begin{pmatrix}
0 \\ \pm 1
\end{pmatrix}=J_b \cdot &\begin{pmatrix}
\pm 1 \\ 0
\end{pmatrix}=0.
\end{align*}
Figures \ref{f:ex2_1} and \ref{f:ex2_2} show the different behavior for $\gamma_1=0.2$,  $\gamma_2=0.1$ and  $\gamma_1=0.1$ and $\gamma_2=0.2$. While in the first case pedestrians can still move to their preferred walking direction, we observe a deadlock in the second case. This confirms the intuitive assumption that stepping aside into the opposite direction as the other group, i.e. $\gamma_1>\gamma_2$, prevents a collision in the next time step. On the other hand, if an individual steps aside in the walking direction of the other group, the initial 'conflict situation' remains unchanged (see Figure \ref{f:sidestepping} for a graphical illustration of such a situation for a red individual).

\begin{figure}[h!]
\begin{center}
\subfigure[$r_T$ at $T=100$]{\includegraphics[height=45mm, width=62mm]{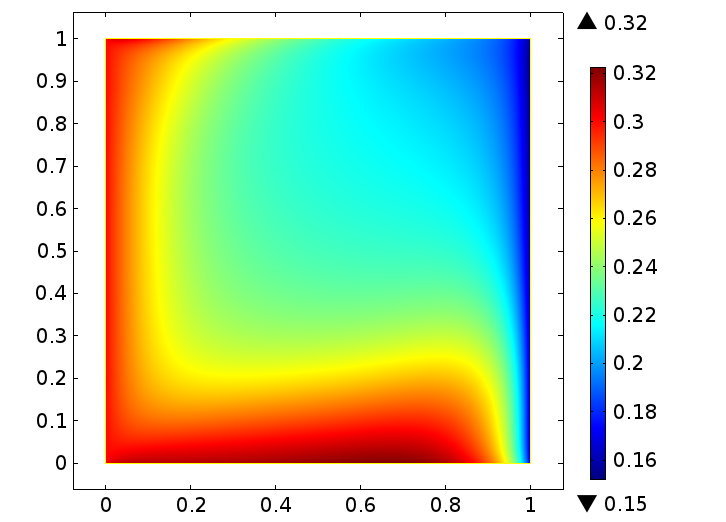}}
\subfigure[$b_T$ at $T=100$]{\includegraphics[height=45mm, width=62mm]{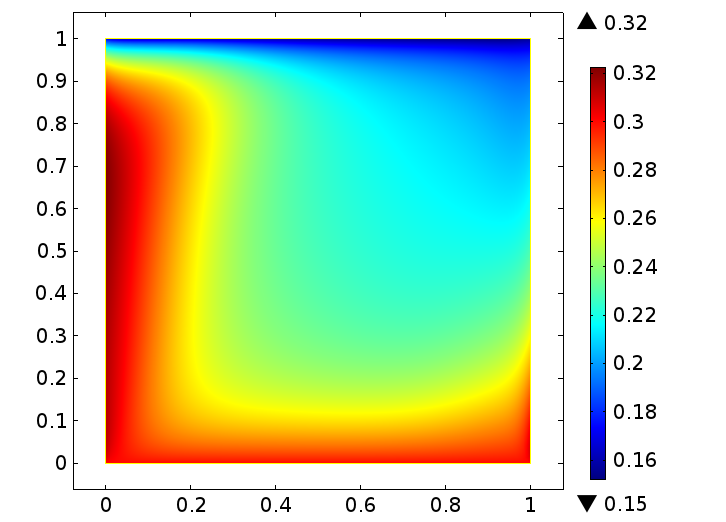}}
\caption{Example II: Particle density for $\gamma_1=0.2$ and $\gamma_2=0.1$.}\label{f:ex2_1}
\end{center}
\end{figure}

\begin{figure}[h!]
\begin{center}
\subfigure[$r_T$ at $T=100$]{\includegraphics[height=45mm, width=62mm]{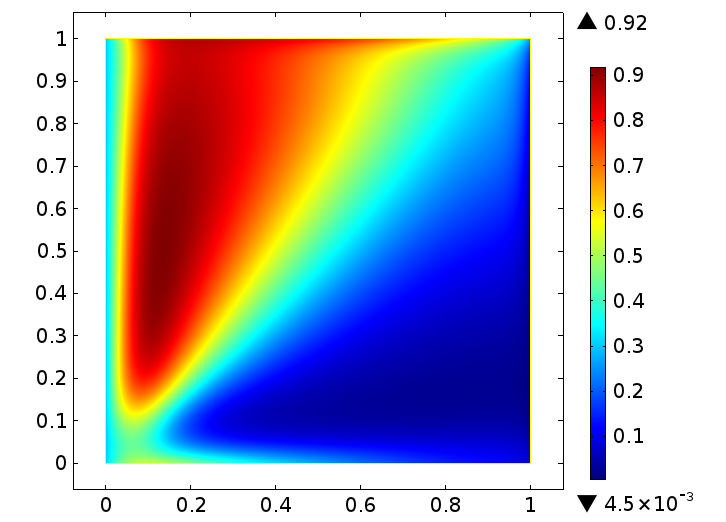}}
\subfigure[$b_T$ at $T=100$]{\includegraphics[height=45mm, width=62mm]{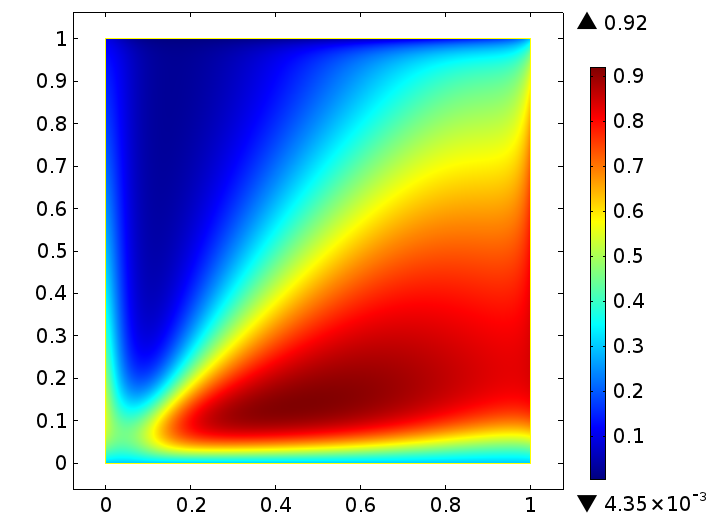}}
\caption{Example II: Particle density for $\gamma_1=0.1$ and $\gamma_2=0.2$.}\label{f:ex2_2}
\end{center}
\end{figure}

\begin{figure}[h!]
\begin{center}
\subfigure[Conflict situation as the red individual wants to go to the right.]{\includegraphics[height=33mm]{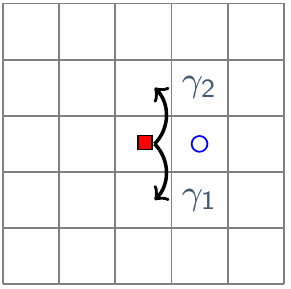}}
\hspace*{0.5cm}
\subfigure[Stepping aside in the same direction as the other group does not resolve the conflict.]{\includegraphics[height=33mm]{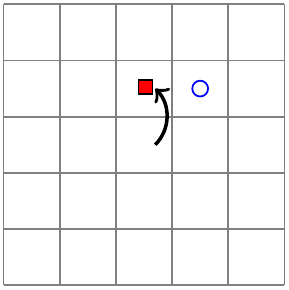}}
\hspace*{0.5cm}
\subfigure[Side stepping in the other direction resolves the situation.]{\includegraphics[height=33mm]{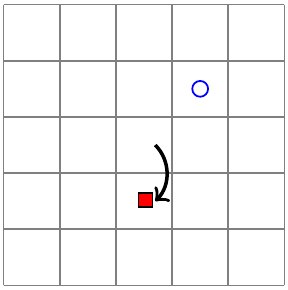}}
\caption{Illustration of a conflict situation for a red individual.}\label{f:sidestepping}
\end{center}
\end{figure}

\subsection{Particle simulations}
In this section we want to compare the macroscopic results presented in Section \ref{numerical_simulations_2D} to the stochastic individual based model introduced in Section \ref{individual_model}. In particular, we perform a particle simulation using Mathematica considering a domain $\Omega$ partitioned into a grid of size $N\times N$ with $P$ particles in total. We define the total density $\rho_\Omega$ by
\[\rho_\Omega:=\frac{P}{N^2}.\]
In each time step we update the position of all particles in a random order.

\subsubsection{Example I}\label{ex_ps:1}
In this example, we set $\alpha=0.6$, $\gamma_0=0.15$, $\gamma_1=0.2$ and $\gamma_2=0.1$, i.e. we consider a small preference to step to one side. Note that this choice of parameters satisfies condition \eqref{CFL}.

We perform two particle simulations for different total densities, namely $\rho_\Omega=0.2$ and $\rho_\Omega=0.5$. This corresponds to a initial random distribution of $2000$ and $5000$ particles on a $100\times 100$ grid. The particle distribution after $500$ time steps in either case is illustrated in Figure \ref{f:ex3}. Whereas for the smaller density the distribution is well mixed, we observe a clear segregation in the case $\rho_\Omega =0.5$. Note that the segregation pattern has a similar structure as in Figure \ref{f:ex1}, the diagonal stripes have the same orientation. Although the particle simulations of the stochastic individual based model cannot directly be compared to the simulations of the macroscopic model, the results let assume that they have a similar behavior. 

\begin{figure}[h!]
\begin{center}
\subfigure[Particle simulation for $\rho_\Omega=0.2$]{\includegraphics[height=50mm, width=50mm]{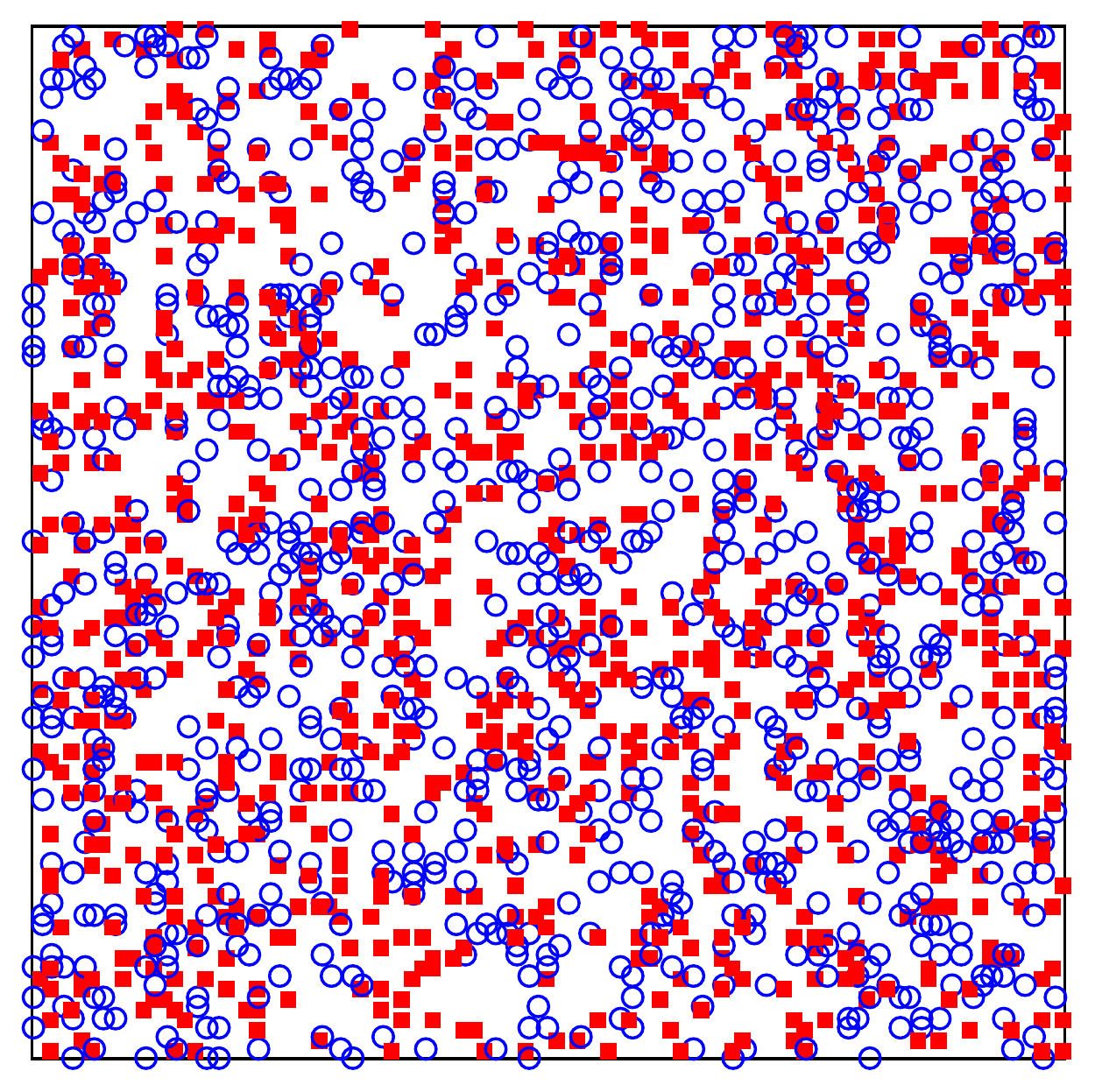}}
\hspace*{0.5cm}
\subfigure[Particle simulation for $\rho_\Omega=0.5$]{\includegraphics[height=50mm, width=50mm]{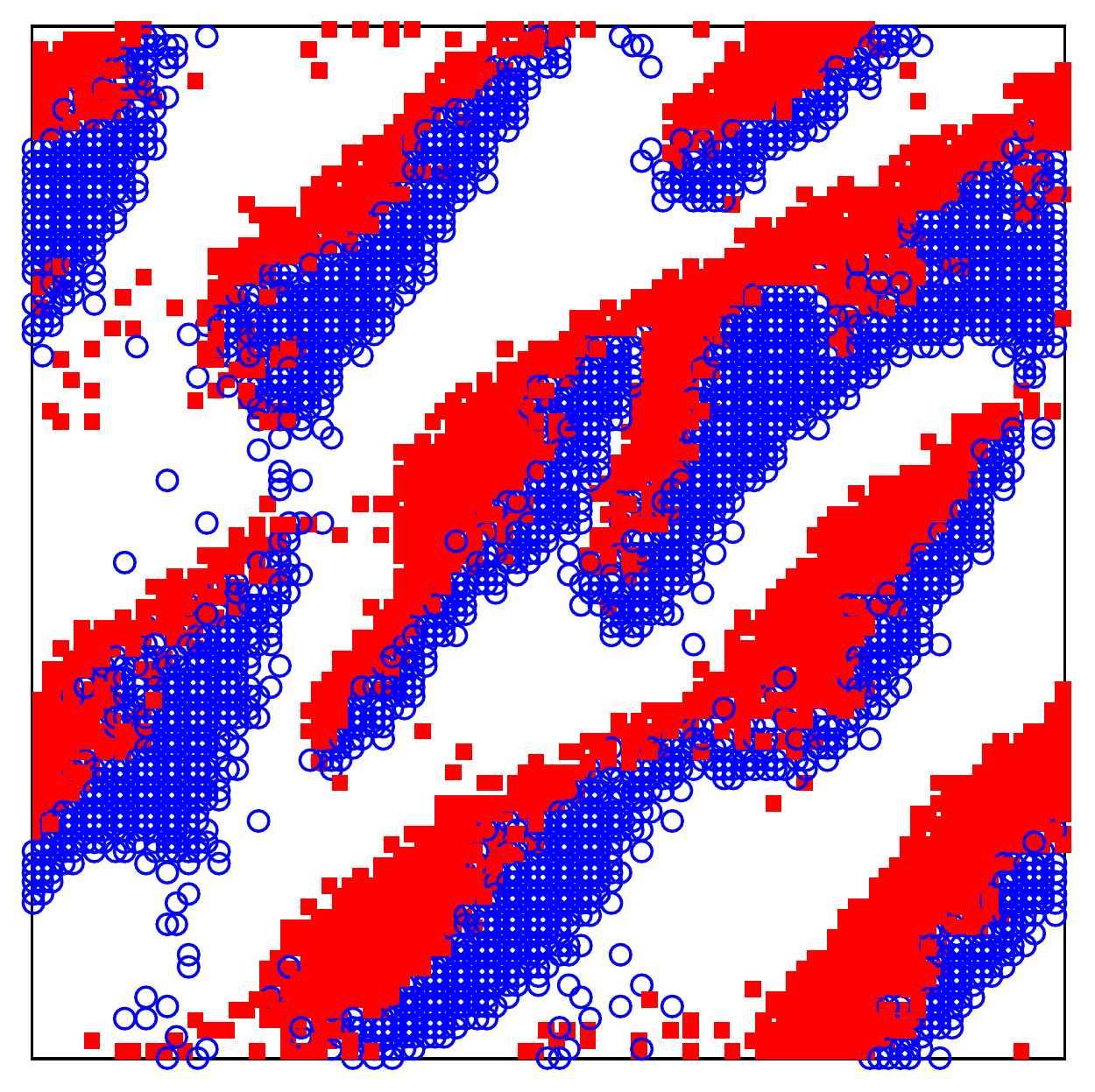}}
\caption{Particle simulation after 500 steps.} \label{f:ex3}
\end{center}
\end{figure}

\subsubsection{Example II}\label{ex_ps:2}
Another interesting feature arises if we choose $\gamma_0=0$ and $\alpha=1$. 
Setting $\gamma_1=0.2$ and $\gamma_2=0.1$ (satisfying condition \eqref{CFL}) and starting with an initial random density of $\rho_\Omega=0.2$, we observe the formation of traveling diagonal wave patterns, see Figure \ref{f:ex4}. These patterns are not stationary and the orientation does not depend on the choice of $\gamma_1$ and $\gamma_2$. We would like to mention that diagonal stripes have been observed in a similar model for pedestrian dynamics (in which pedestrians were not able to step aside), see Ref.~\cite{cividini2013diagonal} and Ref.~\cite{cividini2013crossing}. Note that this particular choice of parameters corresponds to the fact that particles always maintain their walking direction if possible. Only in conflict situations (as illustrated in Figure \ref{f:sidestepping}), the particles try to step aside.

\begin{figure}[h!]
\begin{center}
\includegraphics[height=50mm, width=50mm]{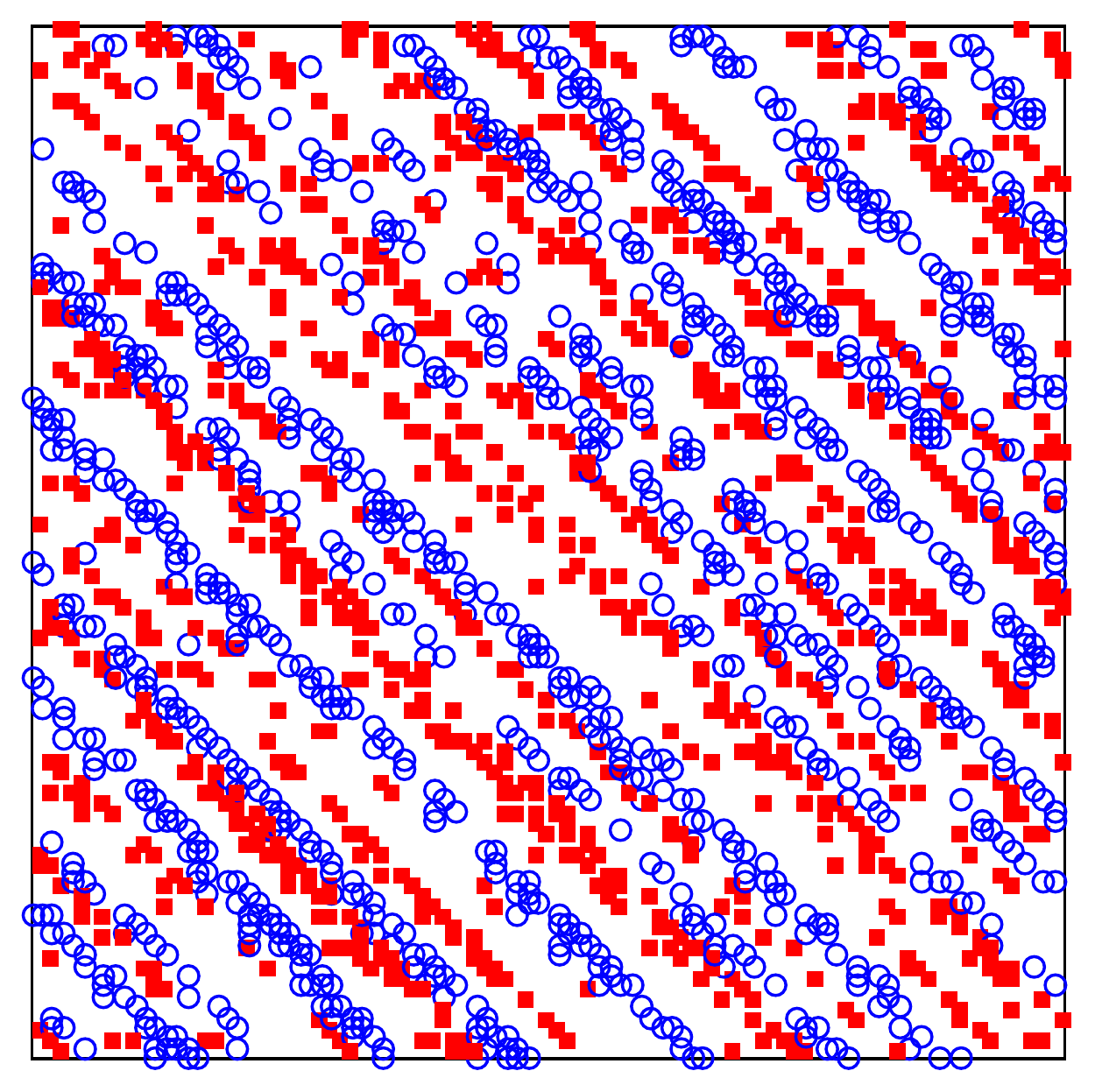}
\caption{Particle simulation after 500 steps.}\label{f:ex4}
\end{center}
\end{figure}

\section{The reduced 1D model}\label{1dmodel}
The preceding simulations show that system \eqref{system} reveals stable as well as unstable regions for different values of $(r,b)\in \overline{\mathcal{M}}$. In order to gain further insights, we consider the first order system \eqref{e:hypsys}, which can be written in the form
\begin{align}\label{e:hypsys2}
\begin{aligned}
 \begin{pmatrix}
\pa_t r\\\pa_t b
\end{pmatrix} &=A\begin{pmatrix}
\pa_x r\\ \pa_x b 
\end{pmatrix}+B \begin{pmatrix}
\pa_y r\\ \pa_y b
\end{pmatrix}, 
\end{aligned}
\end{align}
where 
\[A=\begin{pmatrix}
2r+b-1 & r\vspace*{2mm} \\(\gamma_1-\gamma_2)b(1-2r-b)
 \quad & (\gamma_1-\gamma_2)r(1-r-2b)
\end{pmatrix},\]
and
\[B=\begin{pmatrix}
(\gamma_1-\gamma_2)b(1-2r-b) \quad & (\gamma_1-\gamma_2)r(1-r-2b)\vspace*{2mm}
\\b & r+2b-1 
\end{pmatrix}.\]
The matrices $A$ and $B$ are diagonalizable for $\gamma_1\neq \gamma_2$ and $\rho \leq \frac{1}{2}$. For $\rho>\frac{1}{2}$, both eigenvalues might vanish. Moreover, the eigenvalues of the linear combinations of $A$ and $B$ are not all real in general. Hence, system \eqref{e:hypsys2} is not hyperbolic. Additionally, the system is not genuinely nonlinear. Note that a system of the form
\begin{align*}
 \begin{pmatrix}
\pa_t r\\\pa_t b
\end{pmatrix} &=M(r,b)\begin{pmatrix}
\pa_x r\\ \pa_x b 
\end{pmatrix}
\end{align*}
is called genuinely nonlinear, if $r_k\cdot \nabla_{r,b} \lambda_k\neq 0$ for $k=1,2$, where $r_k$ and $\lambda_k$ denote the right eigenvector and eigenvalue respectively.\\ 
Since the full $2D$ system is very complex, we start the analysis of the simplified $1D$ reduction in the following.\\
In particular, we now focus on the two types of individuals walking in opposite directions described by densities on a line. Note that this situation corresponds to the original problem, if we write the system using coordinates $\left(\frac{x+y}{2},\frac{x-y}{2}\right)$ and study the dynamics of the diagonal patterns. Reduced 1D models for bidirectional pedestrian flows have been studied in Ref.~\cite{AppertRolland2011}, Ref.~\cite{chertock2014pedestrian} and Ref.~\cite{fukui1999self}. The behavior of different models for multi-lane pedestrian flows has been analyzed in Ref. \cite{AppertRolland2011}, which are closly related to the traffic flow models such as the Aw-Rascle and Lighthill-Whitham-Richards (LWR) model. They observe a similar behavior in the proposed first order models, namely the lack of hyperbolicity of the system. Note that this has been reported for classic traffic flow models for $n$ populations in Ref. \cite{benzoni2003}. In Ref. \cite{chertock2014pedestrian} a 1D pedestrian model with slowdown interactions for counterflows in narrow streets is derived. This system is also not hyperbolic and exhibits similar instabilities as we observe in our numerical simulations. The linear stability analysis of both systems gave similar results as we will show in the following.\\
 In our case, the transition rates reduce to
\begin{align*}
\T_r^{i\rightarrow i+1}(r,b)&=\alpha(1-\rho_{i+1}),\\
\T_b^{i\rightarrow i-1}(r,b)&=\alpha(1-\rho_{i-1}),
\end{align*}
which lead to the system 
\begin{align}\label{hypsys1D}
\begin{aligned}
\begin{pmatrix}
\pa_t r\\\pa_t b
\end{pmatrix} =\begin{pmatrix}
-\pa_x ((1-\r)r) \\ \pa_x ((1-\r)b)
\end{pmatrix}&=C(r,b)\begin{pmatrix}
\pa_x r\\ \pa_x b
\end{pmatrix},
\end{aligned}
\end{align}
where 
\[C:=C(r,b)=\begin{pmatrix}
2r+b-1 & r \\ -b & -2b-r+1\\
\end{pmatrix}.\]
The characteristic polynomial of $C$ is 
\[p_C(\lambda)=\lambda^2+\lambda (b-r)-(1-2\rho)(1-\rho)\] and the corresponding eigenvalues are 
\[\lambda_{1,2}=\frac{r-b}{2}\pm \sqrt{\frac{(r-b)^2}{4}+(1-\rho)(1-2\rho)}.\] 
Since the eigenvalues can take complex values system \eqref{hypsys1D} is also not hyperbolic. But we are able to calculate its  hyperbolic regions in $\overline{\mathcal{M}}$ explicitly. Figure \ref{f:hyperbolic_region1D} shows the ellipsoidal region of $r$ and $b$ inside which the $1D$ system is not hyperbolic. 

\begin{figure}[h!]
\begin{center}
\includegraphics[height=50mm, width=50mm]{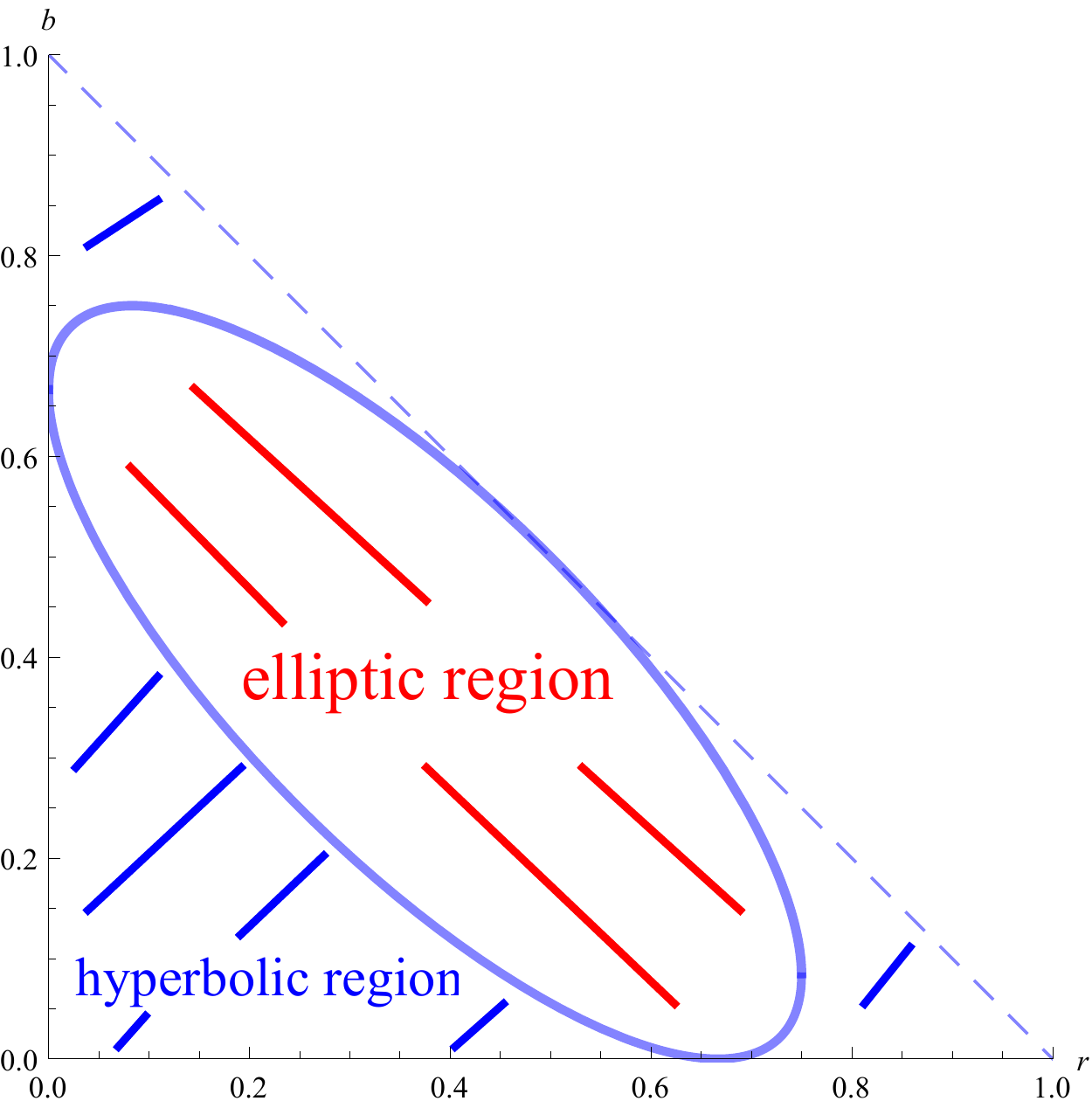}
\caption{Elliptic region of system \eqref{hypsys1D}.} \label{f:hyperbolic_region1D}
\end{center}
\end{figure}

\begin{remark}
In two dimensions, the derivation of the exact hyperbolic region is a more challenging task. Nevertheless, we can check with the help of the corresponding eigenvalues that system \eqref{e:hypsys2} is hyperbolic if $\rho<\frac{1}{2}$.
\end{remark}

\subsection{Linear stability}\label{linearstability}
In this subsection, we take a closer look at the linear stability of equilibrium solutions of system \eqref{hypsys1D} as well as of a regularized, parabolic version of system \eqref{hypsys1D} on $\Omega \times (0,T)$ for $\Omega \subseteq \mathbb{R}$ bounded and we assume periodic boundary conditions. This work is closely related to the linear stability analysis presented in Ref. \cite{AppertRolland2011} and Ref. \cite{chertock2014pedestrian} and allows to determine which equilibrium solutions are stable with respect to small perturbations. Let $(r_\infty,b_\infty)$ be an equilibrium solution, the linearized system around $(r_\infty,b_\infty)$ is
\begin{align}\label{linearized_hyp_1D}
\begin{aligned}
\begin{pmatrix}
\pa_t r\\\pa_t b
\end{pmatrix} =\begin{pmatrix}
-(1-\r_\infty)\pa_x r+r_\infty\pa_x \r \\ (1-\r_\infty)\pa_x b-b_\infty \pa_x \r
\end{pmatrix}.
\end{aligned}
\end{align} 
We look for solutions which are Fourier modes of the form $r=\overline{r}e^{ik\pi x}e^{\lambda t}$ and $b=\overline{b}e^{ik\pi x}e^{\lambda t}$, where $\overline{r},\overline{b}$ are the amplitudes of the mode, $k$ denotes the wave number and $\lambda$ the frequency. Inserting this Fourier ansatz into \eqref{linearized_hyp_1D} leads to the homogeneous linear system
\begin{small}
\begin{align}\label{linearized_hyp_1D_2}
\begin{aligned}
 (C_F-\lambda I)\begin{pmatrix}
r\\ b
\end{pmatrix}=0,
\end{aligned}
\end{align} 
\end{small}
where \[C_F=\begin{pmatrix}
-ik\pi (1-2r_\infty-b_\infty) & ik\pi r_\infty \\ 
-ik\pi b_\infty & ik\pi (1-r_\infty-2b_\infty)
\end{pmatrix}.\]
The system has non-trivial solutions if and only if the determinant of the matrix $C_F-\lambda I$ vanishes resulting in a relation between the frequency $\lambda$ and the wave number $k$. If the real parts of the eigenvalues $\lambda$ are negative for all wave numbers $k\in \mathbb{R}$, system \eqref{linearized_hyp_1D_2} is called asymptotically stable. If one eigenvalue $\lambda$ gets positive for some $k\in \mathbb{R}$, we have instabilities.\\
Evidently, we expect instabilities in the non-hyperbolic region. In the hyperbolic region, the real part of the eigenvalues is zero.  To analyze the linear stability behavior inside the hyperbolic region, we would have to consider also higher order terms. Therefore we analyze and simulate the respective parabolic $1D$ model derived in the same way as the $2D$ model in Section \ref{derivation_macro}. If we add the natural regularization coming from the Taylor expansion, we have
\begin{align}\label{parabolic_1D}
\begin{aligned}
\begin{pmatrix}
\pa_t r\\\pa_t b
\end{pmatrix} =\begin{pmatrix}
-\pa_x ((1-\r)r)+\epsilon (\pa_x((1-b)\pa_x r+r\pa_x b)) \\ \pa_x ((1-\r)b)+\epsilon (\pa_x((1-r)\pa_x b+b\pa_x r))
\end{pmatrix}.
\end{aligned}
\end{align} 
First of all we study the zero-flux stationary solutions of system \eqref{parabolic_1D}, which satisfy
\begin{align}\label{stationary_1D}
\begin{aligned}
\begin{pmatrix}
0\\0
\end{pmatrix} =\begin{pmatrix}
-(1-\r)r+\epsilon ((1-b)\pa_x r+r\pa_x b) \\ (1-\r)b+\epsilon ((1-r)\pa_x b+b\pa_x r)
\end{pmatrix}.
\end{aligned}
\end{align}

\begin{proposition}
Let $(r_S,b_S)$ denote a solution to system \eqref{stationary_1D}. If $0< r_S,b_S,\rho_S<1$, then $\pa_x r_S$ and $\pa_x b_S$ must have different signs at every point $x$.
\end{proposition}
\begin{proof}
Let us assume that $\pa_x r_S$ and $\pa_x b_S$ have the same sign. Then the terms $\epsilon((1-b_S)\pa_x r_S+r_S\pa_x b_S)$ and $\epsilon ((1-r_S)\pa_x b_S+b_S\pa_x r_S)$ would also have the same sign. Since we assume that $(1-\rho_S)r_S$ and $(1-\rho_S)b_S$ are positive, equation \eqref{stationary_1D} can not hold, which leads to a contradiction.
\end{proof}

Linearizing system \eqref{parabolic_1D} around $(r_\infty,b_\infty)$ gives
\begin{align}\label{linearized_parabolic_1D}
\begin{aligned}
\begin{pmatrix}
\pa_t r\\\pa_t b
\end{pmatrix} =\begin{pmatrix}
-(1-\r_\infty)\pa_x r+r_\infty\pa_x \r+\epsilon ((1-b_\infty)\pa_{xx} r+r_\infty\pa_{xx} b) \\ (1-\r_\infty)\pa_x b-b_\infty \pa_x \r+\epsilon ((1-r_\infty)\pa_{xx} b+b_\infty\pa_{xx} r)
\end{pmatrix}.
\end{aligned}
\end{align} 

Inserting again the Fourier ansatz from above into \eqref{linearized_parabolic_1D} leads to the homogeneous linear system

\begin{small}
\begin{align}\label{linearized_parabolic_1D_2}
\begin{aligned}
 (D_F-\lambda I)\begin{pmatrix}
r\\ b
\end{pmatrix}=0,
\end{aligned}
\end{align} 
\end{small}
where \[D_F=\begin{pmatrix}
-ik\pi (1-2r_\infty-b_\infty)-\epsilon k^2\pi^2(1-b_\infty) & -\epsilon k^2\pi^2 r_\infty+ik\pi r_\infty \\ 
-\epsilon k^2\pi^2 b_\infty-ik\pi b_\infty & ik\pi (1-r_\infty-2b_\infty)-\epsilon k^2\pi^2(1-r_\infty) 
\end{pmatrix}.\]
The characteristic polynomial of $D_F$ is
\begin{align*}
p_{D_F}(\lambda)=& \lambda^2-\lambda (ik\pi(r_\infty-b_\infty)-\epsilon k^2\pi^2(2-\rho_\infty))+k^2\pi^2((1-2\rho_\infty)(1-\rho_\infty))\\
&-2i\epsilon k^3\pi^3(1-\rho_\infty)(r_\infty-b_\infty)+\epsilon^2k^4\pi^4(1-\rho_\infty).
\end{align*}
Again the real parts of the eigenvalues $\lambda$ determine the linear stability of system \eqref{linearized_parabolic_1D_2}. The following result has been calculated using Mathematica.\\

\begin{proposition}\label{theorem_linstab}
Let $\epsilon>0$ and let $(r_\infty,b_\infty)$ be such that $0 \leq r_\infty,b_\infty,\rho_\infty \leq 1$. Then system \eqref{linearized_parabolic_1D_2} is linearly stable for $(r_\infty,b_\infty)\notin \mathcal{D}$, where
$\mathcal{D}$ is the area inside the two curves $\gamma_{1,2}: [0,1] \mapsto (r,\min(\frac{-6 + 9 r - 4 r^2}{-9 + 8 r} \pm 4\sqrt{\frac{2 r - 3 r^2 + r^4}{(-9 + 8 r)^2}},1-r)$, see Figure \ref{f:linearstability_region1D}.  \\
Inside the curves $\gamma_{1,2}$, the system is unstable. The diffusion stabilizes the modes corresponding to large wave numbers, i.e. instabilities arise only for modes with
\[k<\frac{\sqrt{\frac{-4 + \r (12 - 8 r^2 + \r (-9 + 8 r))}{ (-2 + \r)^2}}}{\epsilon \pi}.\]
\end{proposition}

\begin{figure}[h!]
\begin{center}
\includegraphics[height=50mm, width=50mm]{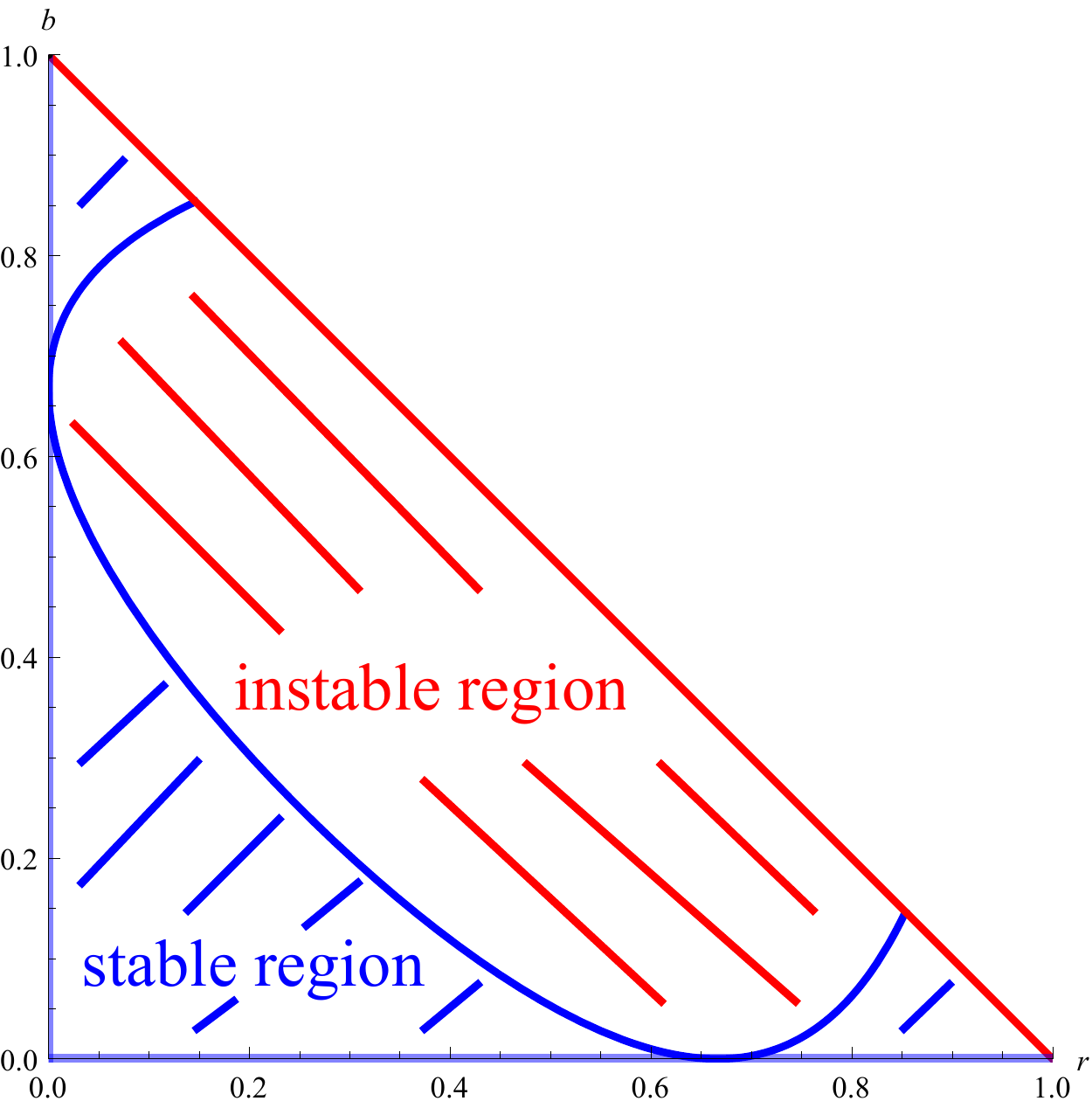}
\caption{Separation of linearly stable and instable region. Note that the blue line still belongs to the stable region and the red line to the unstable one.} \label{f:linearstability_region1D}
\end{center}
\end{figure}

\begin{remark}
If we study the stability of solutions to system \eqref{hypsys1D} in the case of linear diffusion as in Ref. \cite{AppertRolland2011}, we obtain linear stability in the hyperbolic region and instability in the elliptic region, cf. Figure \ref{f:hyperbolic_region1D}.
\end{remark}

\subsection{Local $L^2$-stability resulting from a Lyapunov functional}
In this section, we want to construct a positive entropy functional which can be used to prove local $L^2$-stability of equilibrium solutions to system \eqref{parabolic_1D}. Our results on linear stability already indicate that this will only be possible in a subdomain of $\mathcal{M}$.\\
Note that system \eqref{parabolic_1D} can equivalently be written as
\begin{align*}
\begin{aligned}
\begin{pmatrix}
\pa_t r\\\pa_t b
\end{pmatrix} =\begin{pmatrix}
-\pa_x ((1-\r)r)+\epsilon (\pa_x((1-\rho)\pa_x r+r\pa_x \rho)) \\ \pa_x ((1-\r)b)+\epsilon (\pa_x((1-\rho)\pa_x b+b\pa_x \rho))
\end{pmatrix}.
\end{aligned}
\end{align*}
By adding and subtracting the equations we obtain:
\begin{align*}
\begin{aligned}
\begin{pmatrix}
\pa_t \rho\\\pa_t (r-b)
\end{pmatrix} =\begin{pmatrix}
-\pa_x ((1-\r)(r-b))+\epsilon \pa_x^2\rho \\ -\pa_x ((1-\r)\rho)+\epsilon (\pa_x((1-\rho)\pa_x (r-b)+(r-b)\pa_x \rho))
\end{pmatrix}.
\end{aligned}
\end{align*}
Introducing the new unknowns $\xi:=1-\rho$ and $\eta:=r-b$ gives
\begin{align}\label{parabolic_1D_2}
\begin{aligned}
\begin{pmatrix}
\pa_t \xi \\\pa_t \eta
\end{pmatrix} =\begin{pmatrix}
\pa_x (\eta \xi)+\epsilon \pa_x^2 \xi \\ -\pa_x (\xi(1-\xi))+\epsilon (\pa_x(\xi\pa_x \eta-\eta\pa_x \xi))
\end{pmatrix}.
\end{aligned}
\end{align}

Note that due to mass conservation, $\xi$ and $\eta$ are conserved quantities, i.e.
\begin{align}\label{mass_cons}
\frac{\mathrm{d}}{\mathrm{d}t} \int_\Omega \eta \,dx =\frac{\mathrm{d}}{\mathrm{d}t}\int_\Omega \xi \, dx =0.
\end{align} 

\begin{theorem}
The entropy functional 
\begin{equation}
\label{ent}
\mathcal{J}:=\frac{1}{2}\int_\Omega \eta^2-2(\xi (\log \xi-1)+1)+2(1-\xi)^2 \, dx
\end{equation}
is nonnegative for $(\eta ,\xi) \in [-1,1] \times [0,1]$ and non-increasing in time if
\begin{equation}\label{rest}
\xi \geq \frac{1}{2}+\frac{1}{4\delta} \qquad \textnormal{for  } \ \delta\leq 2. 
\end{equation}
\end{theorem}

\begin{proof}
The nonnegativity of $\mathcal{J}$ follows from
\[-(\xi (\log \xi-1)+1)+(1-\xi)^2 \geq 0 \text{ for } \xi\in [0,1]\,.\] 
Using the property \eqref{mass_cons}, we obtain
\begin{align*}
\frac{\mathrm{d}\mathcal{J}}{\mathrm{d}t}&=\int_\Omega  \eta \pa_t \eta + 2 \xi \pa_t \xi  - \log \xi \pa_t \xi \, dx\\
&=\int_\Omega -\eta \pa_x\xi +2\eta \xi \pa_x \xi -\epsilon(\xi (\pa_x \eta)^2-\eta \pa_x \xi \pa_x \eta)\,dx\\
&\quad +\int_\Omega-2\eta \xi \pa_x \xi -2\epsilon(\pa_x \xi)^2+ \eta \pa_x \xi +\epsilon \frac{1}{\xi} (\pa_x \xi)^2\,dx\\
&= -\epsilon \int_\Omega \xi (\pa_x \eta)^2-\eta \pa_x \xi \pa_x \eta+2(\pa_x \xi)^2- \frac{1}{\xi} (\pa_x \xi)^2\,dx.\\
\end{align*}
Using Young's inequality to estimate the mixed term, i.e.
\[\epsilon \int_\Omega \eta \pa_x \xi \pa_x \eta \,dx =\epsilon \int_\Omega \frac{\eta}{\sqrt{\delta \xi}} \pa_x \xi \sqrt{\delta \xi}\pa_x \eta \,dx \leq \epsilon \int_\Omega \frac{\eta^2}{2\delta\xi}(\pa_x \xi)^2+\frac{\delta\xi}{2}(\pa_x \eta)^2 \,dx, \]
we obtain 
\begin{align}\label{diss}
\frac{\mathrm{d}\mathcal{J}}{\mathrm{d}t}&\leq -\epsilon \int_\Omega \left(2-\left(1+\frac{1}{2\delta}\right)\frac{1}{\xi}\right) (\pa_x \xi) ^2\, dx -\epsilon\left(1-\frac{\delta}{2}\right) \int_\Omega \xi(\pa_x \eta)^2\, dx.
\end{align}
Hence, the entropy functional is decreasing in time for $\delta\leq 2$ and $\xi \geq \frac{1}{2}+\frac{1}{4\delta}$.
\end{proof}
The aim is to show that small perturbations of the equilibrium solutions $(\xi_\infty,\eta_\infty)$ decay in time and we therefore introduce the corresponding relative entropy functional to \eqref{ent}: 
\begin{align}\label{rel.ent}
\begin{split}
\mathcal{J}_{\text{rel}}:=&\frac{1}{2}\int_\Omega (\eta-\eta_\infty)^2-2\xi_\infty\left(\frac{\xi}{\xi_\infty}\left( \log\frac{\xi}{\xi_\infty}-1\right)+1\right)+2(\xi-\xi_\infty)^2\, dx
\end{split}
\end{align}
which satisfies the same entropy dissipation inequality as $\mathcal{J}$ in \eqref{diss}. To guarantee that \eqref{rest} holds for all times $t > 0$, we add another term to the entropy functional. This term allows us to control the $H^1$-norm, and therefore by Sobolev-imbedding, the $L^{\infty}-$ norm of the perturbation $\xi-\xi_\infty$. Differentiating the equation for $\xi$ and testing it with $\pa_x \xi$ we obtain the a priori estimate
\begin{align*}
\begin{aligned}
\frac{1}{2}\frac{\mathrm{d}}{\mathrm{d}t}\int_\Omega (\pa_x \xi)^2\, dx&=\int_\Omega- \pa_x(\eta \xi)\pa_x^2\xi-\epsilon(\pa_x^2\xi)^2\,dx\\
&=\int_\Omega- (\xi\pa_x\eta+\eta\pa_x\xi)\pa_x^2\xi-\epsilon(\pa_x^2\xi)^2\,dx\\  
&\leq\frac{1}{2\epsilon} \int_\Omega \xi^2(\pa_x \eta)^2+\eta^2(\pa_x \xi)^2\,dx +\frac{\epsilon}{2}\int_\Omega (\pa_x^2 \xi)^2 \,dx -\epsilon\int_\Omega (\pa_x^2\xi)^2\,dx\\
&= \frac{1}{2\epsilon}\int_\Omega \xi^2(\pa_x \eta)^2+\eta^2(\pa_x \xi)^2\,dx -\frac{\epsilon}{2}\int_\Omega (\pa_x^2 \xi)^2 \,dx. 
\end{aligned}
\end{align*}
Combining the latter with $\mathcal{J}_{rel}$ we obtain the relative entropy
\begin{align*}
\tilde{\mathcal{J}}_{\text{rel}}:=&\frac{1}{2}\int_\Omega (\eta-\eta_\infty)^2+\alpha\epsilon^2(\pa_x (\xi-\xi_\infty))^2\\
\qquad &-2\xi_\infty\left(\frac{\xi}{\xi_\infty}\left( \log\frac{\xi}{\xi_\infty}-1\right)+1\right)+2(\xi-\xi_\infty)^2\, dx
\end{align*}
for some $\alpha>0$. The relative entropy $\tilde{\mathcal{J}}_{rel}$ is defined in such a way that the nonnegativity is preserved for $\xi_\infty >\frac{1}{2}$ and satisfies 
\begin{align}\label{L2estimate}
{\tilde{\mathcal{J}}}_{\text{rel}}\geq \frac{1}{2}\|\eta-\eta_\infty\|_{L^2(\Omega)}^2+\alpha \epsilon^2 C_p\|\xi-\xi_\infty\|_{L^2(\Omega)}^2,
\end{align}
where $C_p$ denotes the constant resulting from the Poincar\'e inequality
\begin{align*}
\|\xi-\xi_\infty\|_{L^2(\Omega)}^2\leq C_p^{-1} \|\pa_x (\xi-\xi_\infty)\|_{L^2(\Omega)}^2.
\end{align*}
Furthermore, we have (by Sobolev embedding)
\begin{align}\label{Linfestimate}
{\tilde{\mathcal{J}}}_{\text{rel}}\geq C_1(\alpha\epsilon^2)\|\xi-\xi_\infty\|_{H^1(\Omega)}^2\geq C(\alpha\epsilon^2)\|\xi-\xi_\infty\|_\infty^2,
\end{align}
for some positive constants $C_1,C$ both depending on $\alpha \epsilon^2$. 
Moreover 
\[{\tilde{\mathcal{J}}}_{\text{rel}}=0  \iff  (\eta,\xi)=(\eta_\infty,\xi_\infty).\]
Thus, ${\tilde{\mathcal{J}}}_{\text{rel}}$ is a Lyapunov functional and due to its control of the $H^1$-norm of $\xi-\xi_\infty$, it allows us to guarantee the conservation of \eqref{rest} (provided that the initial perturbation is sufficiently small). This leads to the following asymptotic stability result. 

\begin{theorem}[$L^2$- asymptotic stability of equilibria]
Let $\delta \leq 2$ and $(\eta_\infty,\xi_\infty)$ be an equilibrium solution satisfying 
\begin{equation}\label{xiinf}
\xi_\infty=\frac{1}{2}+\frac{1}{4\delta}+\beta
\end{equation} 
for some small $\beta>0$. Moreover, let the initial data $(\eta_0,\xi_0)$ with mean $(\overline{\eta_0},\overline{\xi_0})=(\eta_\infty,\xi_\infty)$ be such that 
\begin{equation}
\tilde{\mathcal{J}}_{\text{rel}}(0)\leq C(\alpha \epsilon^2)\frac{\beta^2}{4}\,,
\end{equation}
where $ C(\alpha \epsilon^2)$ is the positive constant from estimate \eqref{Linfestimate}. Then, the solution $(\eta(t),\xi(t))$ to system \eqref{parabolic_1D} satisfies 
\begin{equation}\label{L2decay}
[\|\eta(t)-\eta_\infty\|_{L^2(\Omega)}+\|\xi(t)-\xi_\infty\|_{L^2(\Omega)}]\rightarrow 0 .
\end{equation}
\end{theorem}

\begin{proof}
As we have seen in \eqref{Linfestimate} the relative entropy ${\tilde{\mathcal{J}}}_{\text{rel}}$ controls the $L^\infty$-norm of the perturbation. Thus, if 
\begin{equation}\label{bound0}
\tilde{\mathcal{J}}_{\text{rel}}\leq C(\alpha \epsilon^2)\frac{\beta^2}{4}\,,\qquad \textnormal{then}\qquad \|\xi-\xi_\infty\|_\infty\leq \frac{\beta}{2}\,,
\end{equation}
such that with \eqref{xiinf} we obtain
\begin{eqnarray}\label{bound}
\xi - \left(\frac12+\frac{1}{4\delta}\right)=\xi-\xi_\infty+\beta\geq \frac{\beta}{2},
\end{eqnarray}
in particular guaranteeing \eqref{rest}. 
Since the initial data is such that ${\tilde{\mathcal{J}}}_{\text{rel}}(0)\leq C(\alpha \epsilon^2)\frac{\beta^2}{4}$, 
we know that initially the relative entropy is decaying, i.e.
\[\frac{d}{dt}\tilde{\mathcal{J}}_{\text{rel}}(0)\leq 0\,.\]
Hence there exists a small time $t_1$ such that $\tilde{\mathcal{J}}_{\text{rel}}(t_1)\leq \tilde{\mathcal{J}}_{\text{rel}}(0)$ and we can repeat the argument. 
Thus \eqref{bound} holds for all times and we have 
\begin{eqnarray}
\frac{\mathrm{d}\mathcal{\tilde{\mathcal{J}}}_{\text{rel}}}{\mathrm{d}t}&\leq& -C_2(\beta\epsilon)\int_\Omega (\pa_x \xi) ^2+ (\pa_x \eta)^2\, dx,\nonumber
\label{entropy_estimate}
\end{eqnarray}
where $C_2$ is a positive constant depending on $\beta \epsilon$.
This, together with \eqref{L2estimate}, implies \eqref{L2decay}. 
\end{proof}

\begin{remark}
We shall emphasize that the lower threshold for $\xi$ corresponds to an upper bound for $\rho$, which is in line with the linear stability result from Section \ref{linearstability}.
\end{remark}

\subsection{Numerical simulations in $1D$} 
We conclude by illustrating the behavior of the one-dimensional model \eqref{parabolic_1D}. All simulations were performed using the COMSOL Multiphysics Package with quadratic finite elements. We consider the domain $\Omega= [0,1]$ (split into $100$ intervals) with periodic boundary conditions  and a BDF method with maximum time step $0.1$.
\subsubsection{Example I}
Let $\epsilon=0.005$ and $(r_\infty,b_\infty)=(0.3,0.3)\in \mathcal{D}$. We recall that $\mathcal{D}$ is the area in which the system is unstable, see Proposition \ref{theorem_linstab}. Starting with a slight perturbation of the equilibrium solutions, i.e. 
\begin{align*}
r_0(x)&=r_\infty+0.02\sin(2\pi x),\\
b_0(x)&=b_\infty-0.02\sin(2\pi x),
\end{align*}
we observe the formation of instabilities as Figure \ref{f:1Dex1} illustrates.
\begin{figure}[h!]
\begin{center}
\subfigure[$r_T$ for $T=100$]{\includegraphics[trim = 0mm 8mm 0mm 0mm, clip,height=38mm, width=41mm]{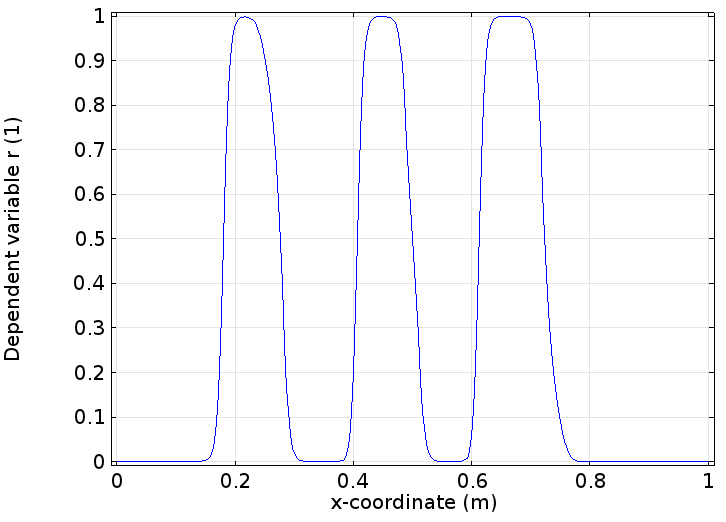}}
\subfigure[$b_T$ for $T=100$]{\includegraphics[trim = 0mm 8mm 0mm 0mm, clip,height=38mm, width=41mm]{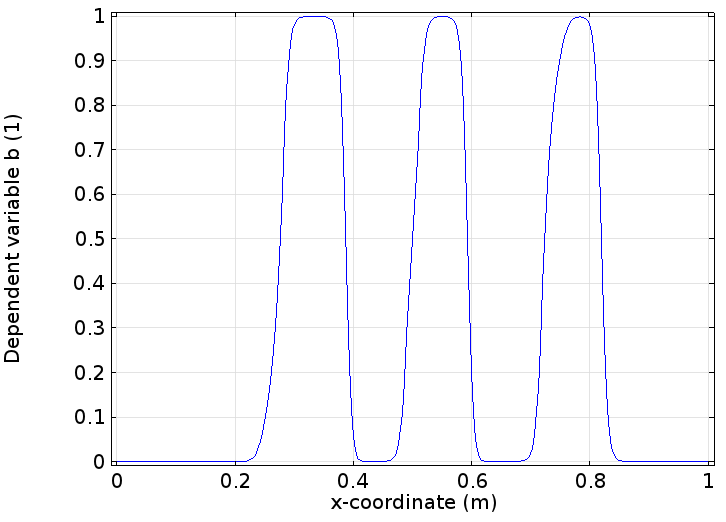}}
\subfigure[$\rho_T$ for $T=100$]{\includegraphics[trim = 0mm 8mm 0mm 0mm, clip,height=38mm, width=41mm]{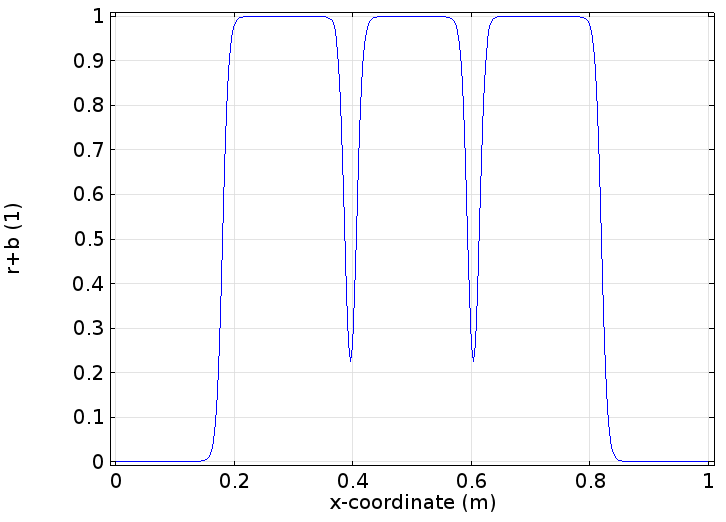}}
\caption{Small perturbations resulting in the formation of shocks in the instable regime.} \label{f:1Dex1}
\end{center}
\end{figure}
Changing the initial values a little, i.e.
\begin{align*}
r_0(x)&=r_\infty+0.01\cos(2\pi x),\\
b_0(x)&=b_\infty-0.01\cos(2\pi x),
\end{align*}
we get a different result, cf. Figure \ref{f:1Dex2}.
\begin{figure}[h!]
\begin{center}
\subfigure[$r_T$ for $T=100$]{\includegraphics[trim = 0mm 8mm 0mm 0mm, clip,height=38mm, width=41mm]{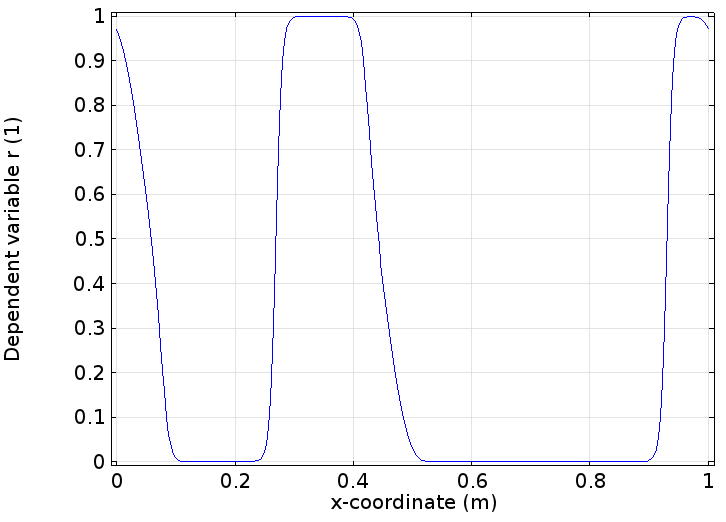}}
\subfigure[$b_T$ for $T=100$]{\includegraphics[trim = 0mm 8mm 0mm 0mm, clip,height=38mm, width=41mm]{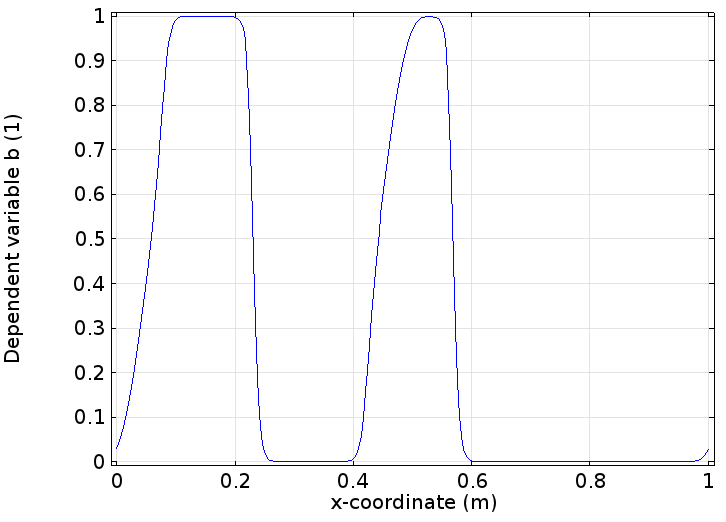}}
\subfigure[$\rho_T$ for $T=100$]{\includegraphics[trim = 0mm 8mm 0mm 0mm, clip,height=38mm, width=41mm]{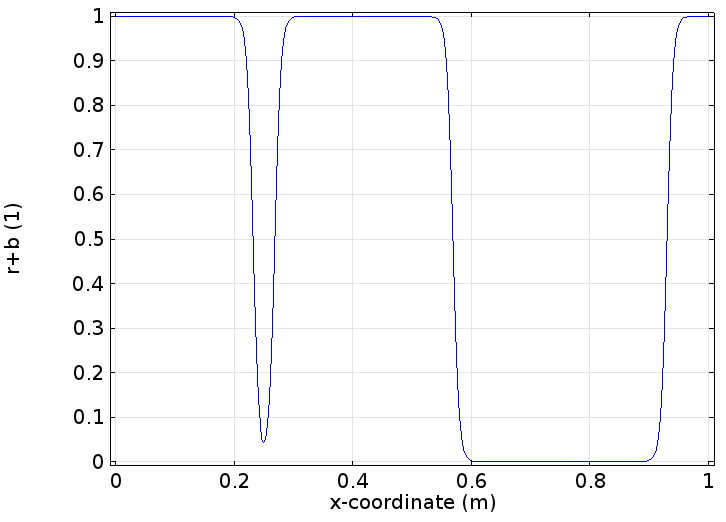}}
\caption{Small perturbations resulting in the formation of shocks in the instable regime.} \label{f:1Dex2}
\end{center}
\end{figure}

\subsubsection{Example II}
If we start outside the unstable region, for example by setting $(r_\infty,b_\infty)=(0.85,0.1)\notin \mathcal{D}$, the solution should go back to its equilibrium value in the case of small perturbations. We set  $\epsilon=0.005$ and  
\begin{align*}
r_0(x)&=r_\infty+0.01\sin(2\pi x),\\
b_0(x)&=b_\infty-0.01\sin(2\pi x),
\end{align*}
and observe the expected behavior in Figure \ref{f:1Dex3}.
\begin{figure}[h!]
\begin{center}
\subfigure[$r_0$ and $r_T$ for $T=1000$]{\includegraphics[trim = 0mm 8mm 0mm 0mm, clip,height=38mm, width=41mm]{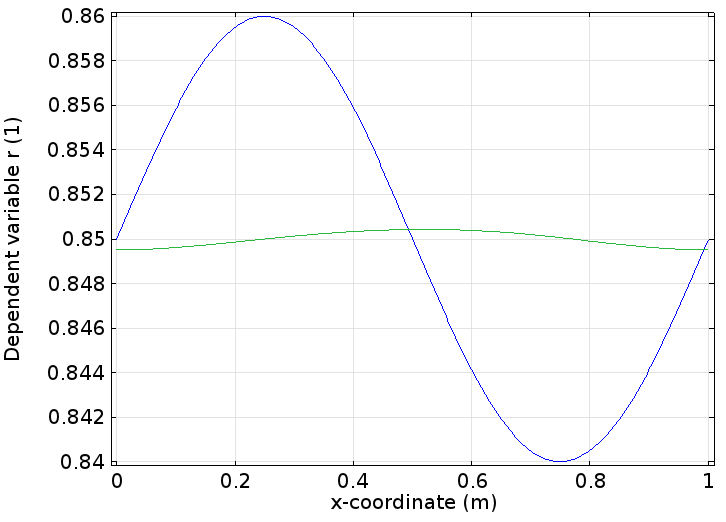}}
\subfigure[$b_0$ and $b_T$ for $T=1000$]{\includegraphics[trim = 0mm 8mm 0mm 0mm, clip,height=38mm, width=41mm]{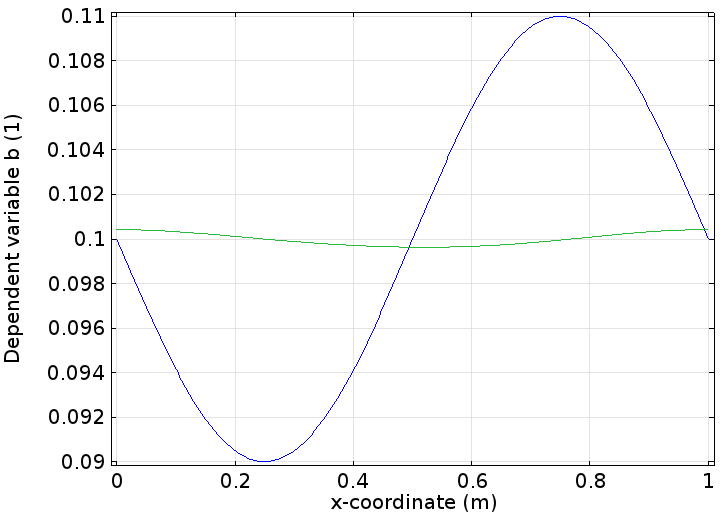}}
\caption{Perturbation smoothing out in the linearly stable regime.} \label{f:1Dex3}
\end{center}
\end{figure}

\begin{remark}
These examples demonstrate that the behavior in the $1D$ case is very similar to the one in $2D$ and therefore suggest the existence of stable and unstable regions also for the $2D$ case. 
\end{remark}

\section*{Acknowledgment}
Sabine Hittmeir thanks the Austrian Science Fund for the support via the Hertha-Firnberg project T-764. Marie-Therese Wolfram and Helene Ranetbauer acknowledge financial support from the Austrian Academy of Sciences \"OAW via the New Frontiers Group NST-001.

\bibliography{crossing}
\bibliographystyle{plain}

\end{document}